\documentclass[12pt]{amsart}

\pdfoutput=1 

\usepackage{amsmath}
\usepackage{amssymb}
\usepackage{amsfonts}
\usepackage{graphicx}
\usepackage{overpic}
\usepackage{fullpage}
\usepackage{color}
\usepackage{enumerate}
\usepackage{thm-restate}
\usepackage{microtype}

\theoremstyle{plain}
\newtheorem{theorem}{Theorem}[section]
\newtheorem{lemma}[theorem]{Lemma}
\newtheorem{proposition}[theorem]{Proposition}
\newtheorem{corollary}[theorem]{Corollary}

\numberwithin{equation}{section}
\numberwithin{figure}{section}

\theoremstyle{definition}
\newtheorem{definition}[theorem]{Definition}
\newtheorem{example}[theorem]{Example}

\newcommand{\G}{\Gamma}

\newcommand{\Z}{\mathbb{Z}}

\newcommand{\CAT}{\operatorname{CAT}}
\newcommand{\cC}{\mathcal{C}}

\newcommand{\cP}{\mathcal{P}}

\newcommand{\s}{\sigma}
\newcommand{\q}{\mathbf{q}}

\newcommand{\cK}{\mathcal{K}}
\newcommand{\cX}{\mathcal{X}}
\newcommand{\cXop}{\mathcal{X}^{\scriptstyle{op}}}
\newcommand{\cY}{\mathcal{Y}}

\newcommand{\bs}{\backslash}
\newcommand{\link}{\operatorname{link}}

\newcommand{\Aut}{\operatorname{Aut}}

\begin{document}

\title{Maximal torsion-free subgroups of certain lattices of hyperbolic buildings and Davis complexes}

\author{William Norledge \and Anne Thomas \and Alina Vdovina}

\begin{abstract}
We give an explicit construction of a maximal torsion-free finite-index subgroup of a certain type of Coxeter group. The subgroup is constructed as the fundamental group of a finite and non-positively curved polygonal complex. First we consider the special case where the universal cover of this polygonal complex is a hyperbolic building, and we construct finite-index embeddings of the fundamental group into certain cocompact lattices of the building. We show that in this special case the fundamental group is an amalgam of surface groups over free groups. We then consider the general case, and construct a finite-index embedding of the fundamental group into the Coxeter group whose Davis complex is the universal cover of the polygonal complex. All of the groups which we embed have minimal index among torsion-free subgroups, and therefore are maximal among torsion-free subgroups.
\end{abstract}

\maketitle

%%%%%%%%%%%%%%%%%%%%%%%
%%%%%%%%%%%%%%%%%%%%%%%
\section{Introduction}\label{sec:Intro}
%%%%%%%%%%%%%%%%%%%%%%%
%%%%%%%%%%%%%%%%%%%%%%%

In the study of locally compact groups, the theory of lattices in the automorphism groups of connected, simply-connected, locally finite polyhedral complexes is a natural extension of the theory of lattices in algebraic groups. By the work of Bruhat-Tits, Ihara, Serre and others, algebraic groups over non-archimedean local fields can be realized as groups of automorphisms of Bruhat-Tits buildings. These buildings can be viewed as certain highly symmetric piecewise Euclidean (simplicial) polyhedral complexes which satisfy the $\CAT(0)$ condition.   

Given a connected, simply-connected, locally finite polyhedral complex $X$, we define the full automorphism group $G=\Aut(X)$ to be the group of cellular isometries of $X$. The group $G$ is a locally compact group in the compact-open topology (equivalently the pointwise convergence topology) which acts properly on $X$ (see~\cite[Chapter I]{KN}).

A subgroup of $G$ is discrete if and only if it acts on $X$ with finite cell stabilizers. A discrete subgroup $\G<G$ is called a \emph{lattice} if $\G\bs G$ has finite Haar measure, and a lattice is \emph{cocompact} if $\G\bs G$ is compact. Note that it may happen that $G$ is discrete, in which case the theory of its lattices is trivial. Lattices in the automorphism groups of locally finite trees (the 1-dimensional case) are called ``tree lattices" and have been widely studied (see the book of Bass-Lubotzky ~\cite{BL}). We refer the reader to ~\cite{FHT} for a recent survey of what is known in higher dimensions. In this paper we study polygonal complexes (the 2-dimensional case). Using covering theory of complexes of groups, we construct minimal index embeddings of a class of torsion-free groups into cocompact lattices of polygonal complexes in two different settings: hyperbolic buildings and Davis complexes.

In our first setting (see Section \ref{sec:CoveringGamma}) we let $L$ be the complete bipartite graph $K_{q_1,q_2}$ on $q_1+q_2$ vertices, where $q_1,q_2\geq 2$. For $m\geq 2$, let $I_{2m,L}$ denote the Bourdon building which is the unique simply connected polygonal complex such that all faces are regular right-angled $2m$-gons and the link at each vertex is $L$. Following Bourdon in~\cite{Bourdon97}, we present this building as the universal cover of a polygon of groups $G(\cP)$ whose fundamental group $\G=\pi_1(G(\cP))$ is a cocompact lattice in $\Aut(I_{2m,L})$. We then construct a covering of the polygon of groups $G(\cP)$ by a certain polygonal complex $X=X_{2m,L}$. We show this covering has $q_1q_2$ sheets, and hence induces an embedding of the torsion-free group $H=\pi_1(X)$ in $\G$ with index $q_1q_2$. By considering torsion in the lattice $\G$, we show that $H$ is a maximal torsion-free subgroup of $\G$.

\begin{restatable}{theorem}{Embeddingingamma}\label{thm:Embeddingingamma}
When $L = K_{q_1,q_2}$ with $q_1,q_2\geq 2$, the group $H$ is an index $q_1q_2$ subgroup of $\G$. Moreover, $H$ is a maximal torsion-free subgroup of $\G$ with minimal index.
\end{restatable}

In our second setting (see Section \ref{sec:CoveringW}) we allow $L$ to be any connected bipartite graph and correspondingly generalize the construction of $X=X_{2m,L}$. Let $W$ denote the Coxeter system which is determined by $L$ as follows: the generators are the vertices, and the product of a pair of generators has order $1$ if they equal, $m$ if they are adjacent, and $\infty$ otherwise. We present $W$ as the fundamental group of a complex of groups $G(\cK)$ whose universal cover is the 2-dimensional Davis complex $\Sigma$ of $W$. The Coxeter group $W$ is naturally a cocompact lattice in $\Aut(\Sigma)$. We construct a covering of $G(\cK)$ by $X$ with $2m$ sheets and show $H$ is a maximal torsion-free subgroup of $W$.  

\begin{restatable}{theorem}{EmbeddinginW}\label{thm:EmbeddinginW}
The group $H$ is an index $2m$ subgroup of $W$.  Moreover, $H$ is a maximal torsion-free subgroup of $W$ with minimal index.
\end{restatable}

A pair of groups are called \emph{commensurable} if they contain finite-index subgroups which are abstractly isomorphic. It is a result of Haglund in~\cite{Haglund} that any two cocompact lattices in $\Aut(I_{2m,L})$, for $m\geq 3$, are commensurable. When $L=K_{q_1,q_2}$, the Davis complex $\Sigma$ can be viewed as the barycentric subdivision of the building $I_{2m,L}$ (see Corollary \ref{daviscomplexisbuilding}). Hence $\G$ and $W$ are both cocompact lattices in $\Aut(I_{2m,L})$, and the subgroup $H$ witnesses their commensurability. 

Finally (see Section~\ref{sec:SurfaceGroups}) we prove that in the first setting we have the following:

\begin{restatable}{theorem}{Amalgam}\label{thm:Amalgam}
When $L = K_{q_1,q_2}$ with $q_1,q_2\geq 2$, the group $H$ is an amalgam of $(q_1 - 1)(q_2 - 1)$ many genus $(m-1)$ surface groups over rank $(m-1)$ free groups. 
\end{restatable} 

In Sections \ref{sec:RAAGs} and \ref{sec:Geoamalgams} we explain how these amalgams generalize both the way that certain right-angled Artin groups can be recognized as amalgams of the free abelian group of rank 2 over infinite cyclic groups, and the geometric amalgams of free groups studied by Lafont in~\cite{Lafont}.

We begin with Section \ref{sec:Preliminaries}, where we collect some preliminary material. In Section \ref{sec:ComplexesLattices} we give a brief exposition of the theory of complexes of groups, which has been tailored to suit our needs. This section also includes some important constructions which are introduced via a sequence of examples. Our main results are proved in Sections \ref{sec:Construction} and \ref{sec:Amalgam}. 

\subsection*{Acknowledgements}

The second author thanks the University of Newcastle (UK) for hosting a visit in December 2015.  The second and third authors were guests of the Forschungsinstitut f\"ur Mathematik (FIM) at ETH Z\"urich in Spring 2016 and we thank the FIM for providing us with a wonderful research environment.

%%%%%%%%%%%%%%%%%%%%%%%
%%%%%%%%%%%%%%%%%%%%%%%
\section{Preliminaries}\label{sec:Preliminaries}
%%%%%%%%%%%%%%%%%%%%%%%
%%%%%%%%%%%%%%%%%%%%%%%

We begin by collecting the basic definitions and results relevant to this paper. We briefly recall polygonal complexes, Bourdon buildings and the Davis complex of a Coxeter group. 

%%%%%%%%%%%%%%%%%%%%%%%%%%%%%%%%
\subsection{Polygonal Complexes}\label{sec:Complexes}
%%%%%%%%%%%%%%%%%%%%%%%%%%%%%%%%

A CW complex is called \emph{regular} if the attaching maps are injective. One of the main attractions of regular CW complexes is the fact that they are rigid with respect to their set of closed cells ordered by inclusion (see~\cite{Bjorner}). We say a regular CW complex has the \emph{intersection property} if the intersection of any two closed cells is either empty or exactly one cell. Equivalently the ordered set of cells has the property that if two cells are bounded below, they have a greatest lower bound.

\begin{definition}\label{def:Polygonalcomplex}
A \emph{polygonal complex} is a connected $2$-dimensional regular CW complex with the intersection property. 
\end{definition}

The prototypical example of a polygonal complex is a connected 2-dimensional simplicial complex. Since we are restricting ourselves to two dimensions, let us adopt the following terminology; we call $0$-cells \emph{vertices}, $1$-cells \emph{edges}, and $2$-cells \emph{faces}. We associate to each vertex $\s$ of a polygonal complex a simplicial graph called its \emph{link}, which we denote by $\link(\s)$. It is the graph whose vertices are edges which intersect $\s$, and whose edges are faces which intersect $\s$. We call a polygonal complex \emph{locally finite} if each of its links is a finite graph. 

The boundary of each face in a polygonal complex is a cycle of at least three edges. Hence faces may be regarded as abstract polygons with at least three sides. Let $k\geq 3$ and $L$ be a finite connected simplicial graph. A \emph{$(k,L)$-complex} is a polygonal complex whose faces are all $k$-gons and whose links are all isomorphic to $L$. A crucial question is the uniqueness of simply-connected $(k,L)$-complexes with respect to a fixed pair $(k,L)$, i.e. to what extent does this local structure determines global structure? In general we don't have uniqueness. For example, for the case where $k\geq 6$ and $L$ is a complete graph on four or more vertices, a continuum of non-isomorphic simply-connected $(k,L)$-complexes was independently constructed by Ballmann-Brin in~\cite{BallBrin}, and by Haglund in~\cite{Haglund2}. Also there are pairs $(k,L)$ for which no simply-connected $(k,L)$-complex exists. 

For $k\geq 4$, $L=K_{q_1,q_2}$ with $q_1,q_2\geq 1$, we have the following: if $k$ is even there is a unique simply-connected $(k,L)$-complex, and if $k$ is odd there is a simply-connected $(k,L)$-complex if and only if $q_1=q_2$, in which case it is unique (see \cite{Swi} and \cite{Wise}). More generally Lazarovich in~\cite{Laz} gives a combinatorial condition on $L$ for which a simply-connected $(k,L)$-complex is unique if it exists.

Polygonal complexes are metrized as follows. We metrize each face as either a spherical, Euclidean or hyperbolic polygon such that the metrics agree on any non-empty intersection of faces. Finally any edges not yet carrying a metric are metrized as intervals of the real line. We equip the polygonal complex with the corresponding quotient pseudometric (see~\cite{BH}, p.65, for details). In general this pseudometric is not a metric, however if the complex has only finitely many isometry types of cells, then this pseudometric is a complete geodesic metric (see~\cite{BH}, p.97). If in addition the complex is locally finite, it follows from the Hopf-Rinow Theorem that the complex is a proper geodesic metric space. In the locally finite case it is also true that the CW topology and metric induced topology will agree. One can always replace a metrized polygonal complex by its barycentric subdivision, which is a metrized simplicial complex (see~\cite{BH}, p.115). From now on we shall assume that a given polygonal complex comes equipped with a metric.  

Finally we remark that the more general ``\emph{polyhedral} complexes" are not required to be regular by most authors. Therefore the notion of a polygonal complex presented here is stronger than the usual notion of a ``2-dimensional polyhedral complex".  

%%%%%%%%%%%%%%%%%%%%%%%%%%%%%%%%
\subsection{Bourdon buildings}\label{sec:Bourdon}
%%%%%%%%%%%%%%%%%%%%%%%%%%%%%%%%

We recall the $2$-dimensional buildings that we will be considering. Let $m, q_1, q_2 \geq 2$ be integers (not necessarily distinct).  Let $\q$ be the $2m$-tuple $\q = (q_1, q_2, \ldots, q_1, q_2)$ with entries alternating between $q_1$ and $q_2$.   
Let $P$ be a Euclidean square, if $m = 2$, and a regular right-angled hyperbolic $2m$-gon, if $m \geq 3$. We define $I_{2m,\q}$ to be the unique simply-connected $(2m,K_{q_1,q_2})$-complex in which each face is metrized as a copy of $P$. In the setting where $L=K_{q_1,q_2}$, we also denote this building by $I_{2m,L}$. 

The complex $I_{2m,\q}$ is often called a \emph{Bourdon building}. We refer the reader to ~\cite{Bourdon97} and~\cite{Bourdon} where they are defined and studied. The most general case of a Bourdon building is $I_{p,\q}$ where $p \geq 4$ and $\q$ is any $p$-tuple of cardinalities at least $2$. The chambers of Bourdon's building are its faces, so each chamber is isometric to $P$. If $m = 2$ then $I_{2m,\q}$ is the product of the $q_1$- and $q_2$-regular trees, and its apartments are copies of the tessellation of the Euclidean plane by squares. If $m \geq 3$ then $I_{2m,\q}$ is not a product space, and its apartments are copies of the tessellation of the hyperbolic plane by regular right-angled $2m$-gons.  

Regarding $P$ as a polygon in the Euclidean plane, if $m = 2$, and in the hyperbolic plane, if $m \geq 3$, let $S_P = \{ s_1, \dots, s_{2m} \}$ be the set of reflections in the sides of $P$, so that $s_i$ and $s_{i+1}$ are the reflections in adjacent sides for $i \in \Z /2m\Z$.  Let $(W_P,S_P)$ be the corresponding right-angled Coxeter system.  That is, $W_P$ has generating set $S_P$ and relations $s_i^2 = 1$ and $(s_i s_{i+1})^2 = 1$ for all $i \in \Z / 2m \Z$.  Then $I_{2m,\q}$ is a right-angled building of type $(W_P,S_P)$, meaning exactly that its apartments are copies of the tessellation of either the Euclidean plane (if $m = 2$) or the hyperbolic plane (if $m \geq 3$) induced by the action of $W_P$.

%%%%%%%%%%%%%%%%%%%%%%%%%%%%%%%%%%
\subsection{Coxeter groups and Davis complexes}\label{sec:Davis}
%%%%%%%%%%%%%%%%%%%%%%%%%%%%%%%%%%

We now describe the Coxeter groups and associated $2$-dimensional Davis complexes that we will be considering. A reference for the material in this section is the book of Davis~\cite{Davis}.
 
Let $m \geq 2$ be an integer.  Let $L$ be a finite, connected, simplicial and bipartite graph with vertex set $$S_L = \{ x_1, \dots, x_{q_1}\} \sqcup \{ y_1, \dots, y_{q_2} \},$$ where $q_1, q_2 \geq 2$ are integers (possibly equal), and every edge of $L$ connects a vertex $x_i$ to a vertex $y_j$. We write $E(L)$ for the edge set of $L$ and $(x_i,y_j)$ for elements of $E(L)$. For example if $E(L)$ contains all possible edges then $L$ is the complete bipartite graph $K_{q_1,q_2}$, but we do not restrict to this case. Figure \ref{fig:Fig1} shows an example of $L$ for the case $q_1=2$, $q_2=3$.

%%%%%%%%%%%%%%%%%%%%%%%%%%%%%%%%%%%%%%%%%
\begin{figure}[t]
	\centering
		\includegraphics{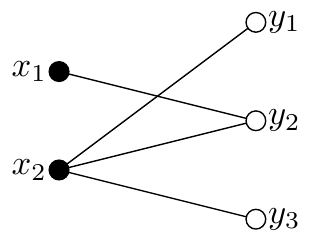}
	\caption{An example of the graph $L$} 
\label{fig:Fig1}
\end{figure}
%%%%%%%%%%%%%%%%%%%%%%%%%%%%%%%%%%%%%%%%%

We define $W=W_{2m,L}$ to be the Coxeter group with generating set $S_L$, and relations $x_i^2 = y_j^2 = 1$ for $1 \leq i \leq q_1$, $1 \leq j \leq q_2$, and $(x_i y_j)^m = 1$ for all $(x_i,y_j) \in E(L)$.  Note that $(W_{2m,L},S_L)$ is a right-angled Coxeter system if and only if $m = 2$. By Moussong's Theorem (see Theorem 12.6.1 of~\cite{Davis}), since $L$ is bipartite the group $W$ is hyperbolic for all $m \geq 3$, and if $m = 2$ is hyperbolic if and only if $L$ contains no embedded $4$-cycles. 

We next recall the construction of the Davis complex $\Sigma = \Sigma_{2m,L}$ for the Coxeter system $(W_{2m,L}, S_L)$.  To simplify notation, put $S = S_L$.  If $T$ is a subset of $S$, the \emph{special subgroup} $W_T$ is the subgroup of $W$ generated by $T$, with $W_\emptyset$ trivial by convention.  For example, each $W_{\{x_i\}}$ and $W_{\{y_j\}}$ is cyclic of order $2$, while if $(x_i,y_j) \in E(L)$  then $W_{\{x_i,y_j\}}$ is the dihedral group of order $2m$.  A \emph{spherical subset} of $S$ is a subset $T \subseteq S$ for which $W_T$ is finite.  In this setting, the spherical subsets of $S$ are $\emptyset$, $\{x_i\}$ for $1 \leq i \leq q_1$, $\{y_j\}$ for $1 \leq j \leq q_2$, and $\{x_i,y_j\}$ whenever $(x_i,y_j) \in E(L)$.

Let $L'$ be the first barycentric subdivision of the graph $L$ and let $K$ be the cone on $L'$.  The $2$-dimensional simplicial complex $K$ is called a \emph{chamber} (note that this chamber is not the same as the chamber for Bourdon's building in Section~\ref{sec:Bourdon} above).  We assign \emph{types} to the vertices of $K$ as follows.  The cone point of $K$ has type $\emptyset$, and each vertex of $K$ which is also a vertex $s$ of $L$ has type $\{s\}$.  Each remaining vertex of $K$ is the midpoint of an edge $(x_i,y_j)$ in $L$, and we assign this vertex of $K$ to have type $\{x_i,y_j\}$.  Observe that this  assignment of types induces a bijection between vertices of $K$ and spherical subsets of $S$, so that the endpoints of each edge in $K$ have types $T' \subsetneq T$. We metrize $K$ as a polygonal complex in the following way: each simplex is metrized as the unique geodesic simplex in either Euclidean or hyperbolic space with angle $\pi/2m$ at the vertex of type $\{x_i,y_j\}$, angle $\pi/2$ at the vertex of type $\{s\}$, and angle $\pi/4$ at the vertex of type $\emptyset$. Observe the simplices are Euclidean only when $m=2$. Figure \ref{fig:Fig2} shows $K$ locally at the edge $(x_2,y_2)$ of the graph in Figure \ref{fig:Fig1}. The vertices have been colored according to their type.

%%%%%%%%%%%%%%%%%%%%%%%%%%%%%%%%%%%%%%%%%
\begin{figure}[t]
	\centering
		\includegraphics{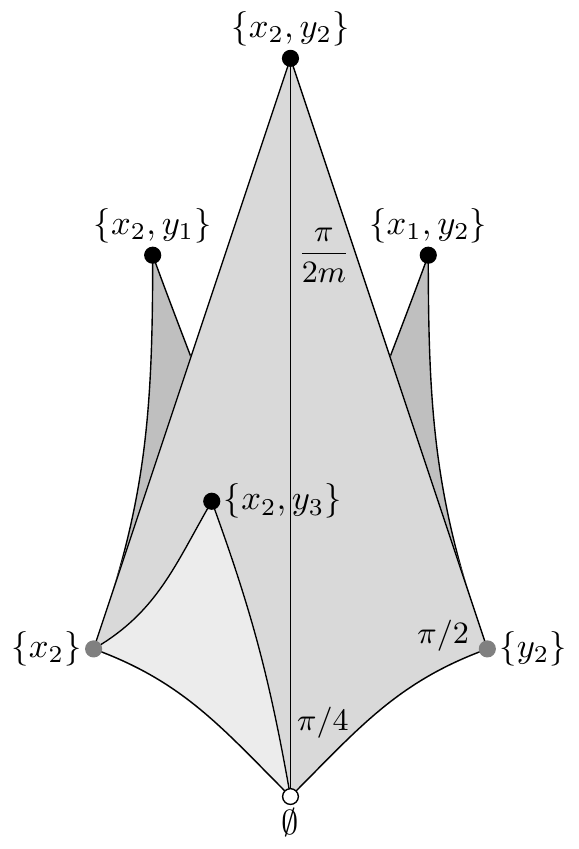}
	\caption{The metrized simplicial complex $K$} 
\label{fig:Fig2}
\end{figure} 
%%%%%%%%%%%%%%%%%%%%%%%%%%%%%%%%%%%%%%%%%

For each $s \in S$, let $K_s$ be the subcomplex of $K$ consisting of all edges of $L'$ which contain the vertex $s$.  In our situation, $K_s$ is the star graph of valence equal to the valence of $s$ in $L$.  The subcomplex $K_s$ is called the \emph{mirror} (of type $s$) of $K$.  Note that two mirrors $K_s$ and $K_{s'}$ intersect, with $K_s \cap K_{s'}$ a point, if and only if the vertices $s$ and $s'$ are adjacent in $L$.  For each point $z \in K$ let $S(z) = \{ s \in S \mid z \in K_s \}$. Then $S(z)$ is empty if and only if $z$ is not in the subcomplex $L'$ of $K$, and otherwise $S(z)$ is either $\{x_i\}$, $\{y_j\}$, or $\{x_i,y_j\}$, with the last case occurring, for the unique point $z = K_{x_i} \cap K_{y_j}$, if and only if $(x_i,y_j) \in E(L)$.  

The \emph{Davis complex} $\Sigma = \Sigma_{2m,L}$ is obtained by, roughly speaking, gluing together $W$-many copies of the chamber $K$ along mirrors.  Formally, $\Sigma$ is the quotient
\[
\Sigma := W \times K / \sim
\]
where $(w,z) \sim (w',z')$ if and only if $z = z'$ and $w^{-1}w'$ is in the special subgroup $W_{S(z)}$.  

In our setting, $\Sigma$ is simplicially isomorphic to the barycentric subdivision of a simply-connected $(2m,L)$-complex, where each $2m$-gon is metrized as a regular right-angled polygon. Thus in the special case that $L$ is the complete bipartite graph $K_{q_1,q_2}$, the Davis complex $\Sigma = \Sigma_{2m,L}$ may be identified with the barycentric subdivision of Bourdon's building $I_{2m,\q}=I_{2m,L}$.  In this identification, the vertices of type $\emptyset$ in $\Sigma$ are the vertices of $I_{2m,\q}$, and the vertices of type $\{x_i,y_j\}$ in $\Sigma$ are the barycenter of faces of $I_{2m,\q}$.  Note that the Coxeter groups $W_{2m,L}$ and $W_P$ are distinct, except if $m = 2$ and $L = K_{2,2}$, in which case $I_{2m,\q}$ is just the tessellation of the Euclidean plane by squares and $\Sigma_{2m,L}$ is the barycentric subdivision of this tessellation.

The assignment of types to the vertices of $K$ induces an assignment of types to the vertices of $\Sigma$, so that two adjacent vertices in $\Sigma$ have types spherical subsets $T' \subsetneq T$.  The group $W$ then has a natural type-preserving left-action on $\Sigma$ with compact quotient $K$, so that the stabilizer of each vertex of $\Sigma$ of type $T$  is a conjugate of the finite group $W_T$. In particular, $W$ acts freely on the set of vertices of $\Sigma$ of type $\emptyset$ (these are the cone points of the copies of $K$ in $\Sigma$). 

%%%%%%%%%%%%%%%%%%%%%%%%%%%%
%%%%%%%%%%%%%%%%%%%%%%%%%%%%
\section{Complexes of groups and construction of lattices}\label{sec:ComplexesLattices}
%%%%%%%%%%%%%%%%%%%%%%%%%%%%
%%%%%%%%%%%%%%%%%%%%%%%%%%% %

In this section we recall the theory of complexes of groups that we will need, mainly following the reference~\cite[Chapter III.$\cC$]{BH}. We will skip many details and give only special cases of definitions.  We also use the sequence of examples in this section to recall the construction of Bourdon's building $I_{2m,\q}$ as the universal cover of a complex of groups whose fundamental group is a lattice in $\Aut(I_{2m,\q})$  (see Section~\ref{sec:Bourdon}), and to realize the Coxeter group $W_{2m,L}$ as the fundamental group of a complex of groups with universal cover the Davis complex $\Sigma_{2m,L}$ (see Section~\ref{sec:Davis}).  The examples in this section are the key information for our proofs in Section~\ref{sec:Construction}, as is Corollary~\ref{cor:Maximality}, which gives us a lower bound on the index of torsion-free subgroups.

%%%%%%%%%%%%%%%%%%%%%%%%%%%%
\subsection{Small categories without loops}\label{sec:Scwols}
%%%%%%%%%%%%%%%%%%%%%%%%%%%%

We will be constructing our complexes of groups over \emph{small categories without loops} (scwols). Scwols will also serve as combinatorial counterparts to polygonal complexes, allowing us to construct coverings of complexes of groups by polygonal complexes in two different settings, that of Bourdon's building $I_{2m,\q}$ and that of the Davis complex $\Sigma_{2m,L}$.

\begin{definition}\label{def:Scwol} 
A \emph{scwol} $\cX$ is the disjoint union of a set $V(\cX)$ of vertices and a set $E(\cX)$ of edges, with edge $a$ oriented from its initial vertex $i(a)$ to its terminal vertex $t(a)$, such that $i(a) \not = t(a)$ for all $a \in E(\cX)$.  A pair of edges $(a,b)$ is \emph{composable} if $i(a)=t(b)$, in which case there is a third edge $ab$, called the \emph{composition} of $a$ and $b$, such that $i(ab)=i(b)$ and $t(ab)=t(a)$, and if both $(a,b)$ and $(b,c)$ are composable then $(ab)c = a(bc)$ (associativity).  
\end{definition}

Scwols can be characterized as ``small categories" (i.e. categories with a \emph{set} of objects and a \emph{set} of morphisms) which don't contain any non-identity endomorphisms or isomorphisms. A scwol is called \emph{thin} if there is at most one edge between each pair of vertices. Observe that the composition of edges in a thin scwol is uniquely determined. Thin scwols are equivalent to partially ordered sets (recall that a partially ordered set $(Q,\leq)$ is naturally a small category by taking $Q$ as its set of objects and including a morphism $\s\rightarrow \tau$ whenever $\s\geq \tau$). The \emph{dimension} of a thin scwol is defined to be one less than the supremum of the lengths of chains in the corresponding partially ordered set.  

One can associate a simplicial complex to a thin scwol by taking the geometric realization of the corresponding partially ordered set (see~\cite{BH}, p.370). The dimension of this simplicial complex is equal to the dimension of the scwol. In the case of a 2-dimensional scwol, if each face of the geometric realization is metrized as a geodesic triangle in either special, Euclidean or hyperbolic space such that the metrics agree on edges, then the geometric realization is naturally a metrized (simplicial) polyhedral complex. 

More generally, for any scwol one can construct the geometric realization of its category theoretic ``nerve" (see~\cite{BH}, p.522). This can then be metrized to give a (simplicial) polyhedral complex (see~\cite{BH}, p.562).

A scwol is called \emph{connected} or \emph{simply-connected} if its geometric realization is respectively connected or simply-connected. From now on all scwols are thin, connected, and at most 2-dimensional.  

\begin{example}\label{egs:Scwols} 
 \begin{enumerate}
\item Let $X$ be a polygonal complex.  We now associate two scwols $\cX$ and $\cXop$ to $X$ such that the geometric realizations of both scwols are equal to the barycentric subdivision of $X$.  The scwol $\cX$ will be used to construct a covering of complexes of groups in the setting of Bourdon's building $I_{2m,\q}$, whereas $\cXop$ will be used in the setting of the Davis complex $\Sigma_{2m,L}$.  Let $X'$ denote the barycentric subdivision of $X$.
\begin{enumerate}
\item The usual way to associate a scwol $\cX$ to $X$ is as follows.  Define $V(\cX) := V(X')$ and $E(\cX) := E(X')$.  The edges of $\cX$ are then oriented from higher-dimensional to lower-dimensional cells in $X$, that is, there is an edge $a \in E(\cX)$ so that $i(a)$ is the barycenter of cell $\s$ and $t(a)$ is the barycenter of cell $\tau$ if and only if $\tau \subsetneq \s$ in $X$.  More concretely, any edge of $\cX$ goes from the barycenter of a face of $X$ to the midpoint of an edge of $X$, or from the midpoint of an edge of $X$ to a vertex of $X$, or from the barycenter of a face of $X$ to a vertex of $X$. This construction naturally metrizes (the geometric realization of) $\cX$ allowing us to recover $X$ from $\cX$. Finally we observe that $\cX$ is equivalent to the set of closed cells of $X$ ordered by inclusion.   
\item For $\cX$ as in (a), we define the \emph{opposite scwol} $\cXop$ to have $V(\cXop) := V(\cX)$ and $E(\cXop) := E(\cX)$, and the orientations of all edges reversed.  That is, for each $a \in E(\cXop) = E(\cX)$, the initial vertex of $a$ in $\cX$ is the terminal vertex of $a$ in $\cXop$, and vice versa.  So in $\cXop$, edges go from lower-dimensional cells to higher-dimensional cells. Similarly this construction naturally equips $\cXop$ with a metric. Finally we observe that $\cXop$ is equivalent to the set of closed cells of $X$ ordered by reverse inclusion.
\end{enumerate}
\item Let $K$ be the chamber for the Coxeter system $(W_{2m,L},S_L)$, as defined in Section~\ref{sec:Davis} above. We associate a scwol $\cK$ to $K$ such that the geometric realization of $\cK$ is equal to $K$ as follows. Let $V(\cK) := V(K)$ and $E(\cK) := E(K)$.  Recall that the endpoints of each edge of $K$ have types $T' \subsetneq T$ where $T'$ and $T$ are spherical subsets of $S = S_L$.  The edges of the scwol $\cK$ are then oriented by inclusion of type, that is, $i(a)$ has type $T'$ and $t(a)$ has type $T$ if and only if $T' \subsetneq T$. Note that a pair of edges $(a,b)$ in $\cK$ is composable if and only if the edge $b$ goes from the cone point of $K$ (which has type $\emptyset$) to a vertex of type either $\{x_i\}$ or $\{y_j\}$, and the edge $a$ goes from $t(b)$ to a vertex of type $\{x_i,y_j\}$ where $(x_i, y_j) \in E(L)$. The metric on $K$ naturally equips $\cK$ with a metric.
\end{enumerate}
\end{example}

Coverings of complexes of groups are defined over the following maps of scwols.  Condition~(3) here restricts the kinds of ``foldings" which are allowed.

\begin{definition}\label{def:MorphismScwols}  Let $\cX$ and $\cY$ be scwols.  A \emph{non-degenerate morphism} $f:\cX \to \cY$ is a map sending $V(\cX)$ to $V(\cY)$ and $E(\cX)$ to $E(\cY)$, so that:
\begin{enumerate}
\item $i(f(a)) = f(i(a))$ and $t(f(a)) = f(t(a))$ for each $a \in E(\cX)$; 
\item $f(ab) = f(a)f(b)$ for each pair of composable edges $(a,b)$ in $\cX$; and
\item for each $\s \in V(\cX)$, the restriction of $f$ to the set of edges $\{a \in E(\cX) \mid i(a) = \s \}$ is a bijection onto the set of edges $\{ a' \in E(\cY) \mid i(a') = f(\s) \}$.
\end{enumerate}
\end{definition}

%%%%%%%%%%%%%%%%%%%%%%%%%%%%
\subsection{Complexes of groups}\label{sec:Complexesofgroups}
%%%%%%%%%%%%%%%%%%%%%%%%%%%%

We now define complexes of groups.

\begin{definition}\label{def:ComplexGroups}
A \emph{complex of groups} $G(\cX)=(G_\sigma, \psi_a)$ over a scwol
$\cX$ is given by: 
\begin{enumerate} \item a group $G_\sigma$ for each
$\sigma \in V(\cX)$, called the \emph{local group} at $\sigma$; and 
\item a monomorphism $\psi_a: G_{i(a)}\rightarrow G_{t(a)}$ along the edge $a$ for each
$a \in E(\cX)$, so that $\psi_{ab} = \psi_a \circ\psi_b$  for each pair of composable edges $(a,b)$.
\end{enumerate}
\end{definition}

\noindent A complex of groups is \emph{trivial} if each local group is trivial. We identify scwols with their corresponding trivial complexes of groups. 

\begin{example}\label{egs:ComplexesGroups}  
We continue notation from Examples~\ref{egs:Scwols}. 
\begin{enumerate}
\item For any polygonal complex $X$ we have the trivial complex of groups $H(\cX)$ over the associated scwol $\cX$, and the trivial complex of groups $H(\cXop)$ over the opposite scwol~$\cXop$.
\item  Let $m, q_1,q_2 \geq 2$ be integers (not necessarily distinct).  Let $P$ be the regular $2m$-gon defined in Section~\ref{sec:Bourdon} with $\cP$ the associated scwol as in Examples~\ref{egs:Scwols}(1a).  
Let $G_1$ be any group of order $q_1$, and $G_2$ any group of order $q_2$.  We now use the groups $G_1$ and $G_2$ to construct a complex of groups $G(\cP) = (G_\s,\psi_a)$ over $\cP$.  Let $\s \in V(\cP)$.   If $\s$ is the barycenter of the face of $P$, the local group $G_\s$ is trivial.  If $\s$ is a vertex of $P$, the local group $G_\s$ is the direct product $G_1 \times G_2$.  The remaining $\s$ are the midpoints of edges of $P$, and the local groups for these $2m$ edges alternate between $G_1$ and $G_2$, so that at each vertex of $P$, one of the adjacent local groups is $G_1$ and the other is $G_2$.
All monomorphisms of local groups are the natural inclusions.
\item Let $\cK$ be the scwol associated to the chamber $K$ for the Coxeter system $(W_{2m,L},S_L)$ as in Examples~\ref{egs:Scwols}(2).  We construct a complex of groups $G(\cK)$ over $\cK$ as follows.  Let $\s \in V(\cK)$.  Then $\s$ has type a spherical subset $T \subseteq S$, and we define $G_\s$ to be the (finite) special subgroup $W_T$.  Note that the cone point has  trivial group, and all other local groups are either cyclic of order $2$ or dihedral of order $2m$.  All monomorphisms are the natural inclusions.
\end{enumerate}
\end{example}

If $L$ is the graph shown in Figure \ref{fig:Fig1}, then Figure \ref{fig:Fig3} shows $G(\cK)$ locally at the (image of the) edge $(x_2,y_2)$ in $\cK$. The dihedral group of order $2m$ is denoted by $D_m$ and the cyclic group of order $2$ is denoted by $Z_2$.

%%%%%%%%%%%%%%%%%%%%%%%%%%%%%%%%%%%%%%%%%
\begin{figure}[t]
	\centering
		\includegraphics{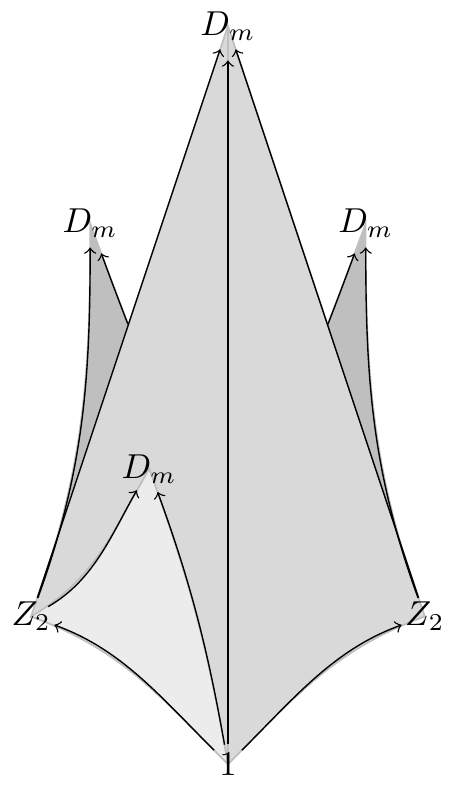}
	\caption{The complex of groups $G(\cK)$ over $\cK$} 
\label{fig:Fig3}
\end{figure}
%%%%%%%%%%%%%%%%%%%%%%%%%%%%%%%%%%%%%%%%%

We refer the reader to~\cite{BH} for the general definition of the \emph{fundamental group} $\pi_1(G(\cX))$ of a complex of groups $G(\cX)$.  We will only need the following examples, where we continue notation from Examples~\ref{egs:ComplexesGroups}.

\begin{example}\label{egs:FundGroup}  
\begin{enumerate}
\item If a polygonal complex $X$ has (topological) fundamental group $H$, then the trivial complexes of groups $H(\cX)$ and $H(\cXop)$ have fundamental group  $H$ as well.
\item The fundamental group of $G(\cP)$ has presentation
\[
\pi_1(G(\cP)) = \langle G_{1,1},\dots,G_{1,m}, G_{2,1},\dots, G_{2,m} \mid [G_{1,k},G_{2,k}] = [G_{2,k},G_{1,k+1}] = 1 \rangle
\]
where for $i = 1,2$ and $k  \in \Z/m\Z$, each $G_{i,k}$ is isomorphic to $G_i$.  (In this presentation, the relations within each group $G_{i,k}$ are included implicitly.)  The commutator relations mean that local groups on adjacent edges of $P$ commute with each other in $\pi_1(G(\cP))$.  Thus $\pi_1(G(\cP))$ may be viewed as a graph product of groups, where the underlying graph is a $2m$-cycle and the groups $G_1$ and $G_2$ are placed on alternate vertices in this cycle.
\item The fundamental group of $G(\cK)$ is the Coxeter group $W_{2m,L}$.
\end{enumerate}
\end{example}

Every complex of groups has a universal cover which is a (possibly non-trivial) complex of groups with a trivial fundamental group. Induced by the action of a group $\G$ on a scwol $\cX$ is the quotient complex of groups $\G\ltimes \cX$. If $\cX$ is simply connected then $\cX$ is the universal cover of $\G\ltimes \cX$ and $\pi_1(\G\ltimes \cX)\cong \G$. A complex of groups $G(\cX) = (G_\s,\psi_a)$ is called \emph{developable} if it arises as a quotient complex of groups in this way. Unlike graphs of groups, complexes of groups are not in general  developable.  The examples $H(\cX)$, $H(\cXop)$, $G(\cP)$ and $G(\cK)$ above are all developable. This follows from the fact that they are all non-positively curved (see~\cite{Hae}). For the general construction of the complex of groups induced by a group acting on a scwol, see~\cite{BH}.  

Conversely, if the universal cover $\widetilde{G(\cX)}$ of a complex of groups $G(\cX)$ is trivial, and hence a simply-connected scwol, then $\widetilde{G(\cX)}$ is naturally equipped with an action of $\pi_1(G(\cX))$ such that the complex of groups induced by this action is (isomorphic to) $G(\cX)$. Hence $G(\cX)$ is developable.

It can be shown that the existence of a trivial universal cover for a complex of groups $G(\cX)$ is equivalent to the following: for all $\s \in V(\cX)$, the local group $G_\s$ embeds in the fundamental group $\pi_1(G(\cX))$. 

If $\cX$ is metrized, then $\widetilde{G(\cX)}$ is naturally metrized by developing the metric equivariantly. Conversely the complex of groups induced by a group acting on a metrized scwol is naturally metrized. 

We will only need the following examples.

\begin{example}\label{egs:UniversalCover} 
\begin{enumerate}
\item  Let $\widetilde X$ be the simply-connected polygonal complex which is the (classical) universal cover of $X$.  Then the universal cover of the trivial complex of groups $H(\cX)$ is the scwol associated to $\widetilde X$ as in Examples~\ref{egs:Scwols}(1a), and the universal cover of the trivial complex of groups $H(\cXop)$ is the opposite scwol associated to $\widetilde X$ as in Examples~\ref{egs:Scwols}(1b).  The complexes of groups $H(\cX)$ and $H(\cXop)$ are induced by the free action of $H = \pi_1(X)$ on $\widetilde X$, and $X$ is the quotient space $H \bs \widetilde X$.  
\item  The universal cover of $G(\cP)$ is the scwol associated to the unique simply-connected $(2m,L)$-complex with $L = K_{q_1,q_2}$.  Hence the  universal cover of $G(\cP)$ is (the scwol associated to) Bourdon's building $I_{2m,\q}$, and the complex of groups $G(\cP)$ is induced by the action of $\pi_1(G(\cP))$ on $I_{2m,\q}$.  It follows that $\pi_1(G(\cP))$ acts on $I_{2m,\q}$ with compact quotient $P$, so that the stabilizer of each face of $I_{2m,\q}$ is trivial, the stabilizer of each edge of $I_{2m,\q}$ is isomorphic to either $G_1$ or $G_2$, and the stabilizer of each vertex of $I_{2m,\q}$ is isomorphic to $G_1 \times G_2$.
\item The geometric realization of the universal cover $\widetilde{G(\cK)}$ of $G(\cK)$ is the Davis complex $\Sigma_{2m,L}$. The action of $W_{2m,L}$ on $\Sigma_{2m,L}$ can naturally be regarded as an action of $W_{2m,L}$ on $\widetilde{G(\cK)}$, and $G(\cK)$ is the induced complex of groups. 
\end{enumerate}
\end{example}
 
A developable complex of groups is \emph{faithful} if its fundamental group acts effectively on its universal cover.  A sufficient condition for faithfulness of a developable complex of groups $G(\cX) = (G_\s, \psi_a)$ is that one of the local groups $G_\s$ be trivial.  Thus all of the examples $H(\cX)$, $H(\cXop)$, $G(\cP)$ and $G(\cK)$ we have been discussing are faithful.  If $G(\cX)$ is developable and faithful, with universal cover $\cY:=\widetilde{G(\cX)}$, we may identify $\pi_1(G(\cX))$ with a subgroup of $\Aut(\cY)$. 

Let $Y$ be either a simply-connected locally finite polygonal complex or the Davis complex $\Sigma_{2m,L}$. If $Y$ is a polygonal complex, we pair it with the associated scwol in the sense of Examples~\ref{egs:Scwols}(1a), and if $Y$ is the Davis complex $\Sigma_{2m,L}$ we pair it with $\widetilde{G(\cK)}$. Suppose the universal cover $\cY$ is the scwol paired with $Y$. Identify the automorphisms of $\cY$ with their corresponding automorphisms of $Y$, allowing $\pi_1(G(\cX))$ to be identified with a subgroup of $\Aut(Y)$. Then $\pi_1(G(\cX))$ acts cocompactly on $Y$ if and only if $\cX$ is a finite scwol.  Also, $\pi_1(G(\cX))$ is a discrete subgroup of $\Aut(Y)$ if and only if all local groups in $G(\cX)$ are finite.  It follows that if $\cX$ is finite and all local groups of $G(\cX)$ are finite, we may regard $\pi_1(G(\cX))$ as a cocompact lattice in $\Aut(Y)$. In particular we have the following:

\begin{example}\label{egs:Lattices} 
\begin{enumerate}
\item  Let $X$ be a finite polygonal complex with fundamental group $H$ and universal cover $\widetilde X$.  Then $H = \pi_1(H(\cX)) = \pi_1(H(\cXop))$ is a cocompact lattice in $\Aut(\widetilde X)$.  
\item  From now on, write $\G_{2m}(G_1,G_2)$ or simply $\G$ for the fundamental group $\pi_1(G(\cP))$ with presentation given in Examples~\ref{egs:FundGroup}(2) above.  Then $\G$ is a cocompact lattice in $\Aut(I_{2m,\q})$. 
\item The Coxeter group $W_{2m,L} = \pi_1(G(\cK))$ is a cocompact lattice in $\Aut(\Sigma_{2m,L})$. It is known that $\Aut(\Sigma_{2m,L})$ is non-discrete if $L$ has a non-trivial automorphism which fixes the star of a vertex (see~\cite{HP}). For example the graph in Figure \ref{fig:Fig1} has a non-trivial automorphism which fixes the star of the vertex $x_1$.
\end{enumerate}
\end{example}

%%%%%%%%%%%%%%%%%%%%%%%%%%%%
\subsection{Coverings of complexes of groups}\label{sec:Coveringsofcomplexesofgroups}
%%%%%%%%%%%%%%%%%%%%%%%%%%%%

We now define a covering of complexes of groups.  We give this definition only in the special case that we will need.

\begin{definition}\label{def:Covering} 
Let $f: \cX\to \cY$ be a non-degenerate morphism of scwols.  Let $H(\cX)$ be a trivial complex of groups and let $G(\cY) = (G_\sigma, \psi_a)$ be a complex of groups.   A \emph{covering of complexes of groups} $\Phi: H(\cX) \to G(\cY)$ over $f$ consists of an element $\phi(a) \in G_{f(t(a))}$ for each $a \in E(\cX)$, such that: 
\begin{enumerate}
\item\label{i:commuting}  for all pairs of
composable edges $(a,b)$ in $E(\cX)$, $\phi(ab) = \phi(a) \,\psi_{f(a)}(\phi(b))$; 
and 
\item for each $\sigma \in V(\cX)$ and each $b \in E(\cY)$ such that
$t(b) = f(\sigma)$, the map \[  \Phi_{\s/b}: \{ a \in E(\cX) \mid f(a)  = b \mbox{ and } t(a)=\sigma\}  \to G_{f(\sigma)} / \psi_b(G_{i(b)})\] induced by $a \mapsto \phi(a)$ is a bijection.
\end{enumerate} 
\end{definition}

Let $\Phi:H(\cX) \to G(\cY)$ be a covering of complexes of groups as in Definition~\ref{def:Covering}.  Suppose that $\cX$ and $\cY$ are finite scwols and that all local groups in $G(\cY)$ are finite.  Let $\tau \in V(\cY)$.  The \emph{number of sheets} of the covering $\Phi$ is the positive integer
\[
n := \sum_{\s \in f^{-1}(\tau)} \frac{|G_\tau|}{|H_\sigma|} =  |f^{-1}(\tau)|\cdot |G_\tau|.
\]
This definition is independent of the choice of $\tau \in V(\cY)$, since $\cY$ is connected, and the last equality holds since $H(\cX)$ is a complex of trivial groups.   If $\Phi$ has $n$ sheets we may say that $\Phi$ is an \emph{$n$-sheeted} covering.  In particular, if some local group $G_\tau$ is trivial then $\Phi$ is $n$-sheeted where $n = |f^{-1}(\tau)|$.

The next two theorems are special cases of results on functoriality of coverings which are implicit in~\cite{BH}, and stated and proved explicitly in~\cite{LT}.

\begin{theorem}\label{thm:Coverings} 
Let $H(\cX)$ be a trivial complex of groups and $G(\cY)$ be a complex of groups, where $\cX$ and $\cY$ are finite scwols.  Suppose both complexes of groups are developable and that there is an $n$-sheeted covering of complexes of groups $\Phi:H(\cX) \to G(\cY)$.  Put $H = \pi_1(H(\cX))$ and $G = \pi_1(G(\cY))$.  Then $\Phi$ induces an embedding of $H$ as an index $n$ subgroup of $G$ 
and an equivariant isomorphism of universal covers $\widetilde{H(\cX)} \longrightarrow
\widetilde{G(\cY)}$. 
\end{theorem}

\begin{theorem}\label{thm:Subgroups}  
Suppose $G(\cY) = (G_\s,\psi_a)$ is a developable complex of finite groups over a finite scwol $\cY$.  Let $G = \pi_1(G(\cY))$.  Then for any torsion-free index $n$ subgroup $H$ of $G$, there is an $n$-sheeted covering of complexes of groups $\Phi:H(\cX) \to G(\cY)$, where $H(\cX)$ is a trivial complex of groups over a finite scwol $\cX$ and $\pi_1(H(\cX)) = H$.
\end{theorem}

Using Theorem~\ref{thm:Subgroups}, we can obtain lower bounds on the index of a torsion-free subgroup, as follows.

\begin{corollary}\label{cor:Sheets}  
Let $G(\cY)$ be as in the statement of Theorem~\ref{thm:Subgroups}.  Suppose $H$ is a torsion-free index $n$ subgroup of $G = \pi_1(G(\cY))$.  Then $n \geq |G_\tau|$ for all $\tau \in V(\cY)$.
\end{corollary}

\begin{proof}  
Let $\tau \in V(\cY)$ and let $\Phi:H(\cX) \to G(\cY)$ be an $n$-sheeted covering corresponding to $H$, as guaranteed by Theorem~\ref{thm:Subgroups}.  Then by definition of the number of sheets, $n = |f^{-1}(\tau)| \cdot |G_\tau|$.  Now $f$ is surjective so $|f^{-1}(\tau)| \geq 1$, and the result follows.
\end{proof}

Applying this to the groups considered in the sequence of examples in this section, we have:

\begin{corollary}\label{cor:Maximality}  
Let $\G_{2m}(G_1,G_2)$ and $W_{2m,L}$ be the groups realized above as fundamental groups of the complexes of groups $G(\cP)$ and $G(\cK)$, respectively.  Then:
\begin{enumerate}
\item Any torsion-free finite-index subgroup of $\G_{2m}(G_1,G_2)$ has index at least $q_1 q_2$.  
\item Any torsion-free finite-index subgroup of $W_{2m,L}$ has index at least $2m$.
\end{enumerate}
\end{corollary}

\begin{proof}  
The complexes of groups $G(\cP)$ and $G(\cK)$ are both developable. The complex of groups $G(\cP)$ has local groups including $G_1 \times G_2$ of order $q_1 q_2$, and the complex of groups $G(\cK)$ has local groups including the dihedral group of order $2m$.
\end{proof}

%%%%%%%%%%%%%%%%%%%%%%%%%%%%%%%%%%%%%%%%%
%%%%%%%%%%%%%%%%%%%%%%%%%%%%%%%%%%%%%%%%%
\section{Embeddings of maximal torsion-free finite-index subgroup}\label{sec:Construction}
%%%%%%%%%%%%%%%%%%%%%%%%%%%%%%%%%%%%%%%%%
%%%%%%%%%%%%%%%%%%%%%%%%%%%%%%%%%%%%%%%%%

In this section we construct two coverings of complexes of groups, which induce finite-index embeddings of a maximal torsion-free group in the lattices constructed in Section ~\ref{sec:Complexesofgroups}. We continue notation from Section~\ref{sec:ComplexesLattices}.

%%%%%%%%%%%%%%%%%%%%%%%
\subsection{Construction of the torsion-free group}\label{sec:H}
%%%%%%%%%%%%%%%%%%%%%%%

First we construct the torsion-free group $H$ as the (topological) fundamental group of a polygonal complex $X$. Let us introduce some notation that will be used throughout Section~\ref{sec:Construction}. Let $L$ continue to denote a finite, connected, simplicial and bipartite graph as in Section~\ref{sec:Davis}. For each edge $(x_i,y_j) \in E(L)$, let $P_{ij}$ be a copy of the $2m$-gon $P$ described in Section~\ref{sec:Bourdon}. If $L = K_{q_1,q_2}$ there are $q_1 q_2$ such polygons, and otherwise there are strictly fewer than $q_1 q_2$ of them.  Now orient the edges of each $P_{ij}$ cyclically, and label the edges in the resulting $2m$-cycle going around the boundary of $P_{ij}$ by the word $x_i^1 y_j^1 x_i^2 y_j^2 \dots x_i^m y_j^m$. 

For each $k \in \Z/m\Z$, we label the vertex of $P_{ij}$ with incoming edge $x_i^k$ and outgoing edge $y_j^k$ by $u_{ij}^k$, and we label the vertex of $P_{ij}$ with incoming edge $y_j^k$ and outgoing edge $x_i^{k+1}$ by $v_{ij}^k$. In $P_{ij}$, the link of each $u_{ij}^k$ and each $v_{ij}^k$ is a single edge, which can be identified with the edge $(x_i,y_j) \in E(L)$. Figure \ref{fig:Fig4} shows the labeling of $P_{ij}$ for the case $m=3$.

%%%%%%%%%%%%%%%%%%%%%%%%%%%%%%%%%%%%%%%%%
\begin{figure}[t]
	\centering
		\includegraphics{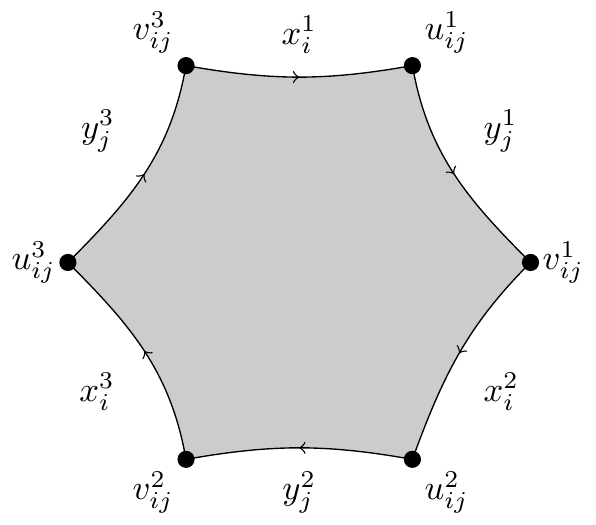}
	\caption{The labeling of $P_{ij}$} 
\label{fig:Fig4}
\end{figure}
%%%%%%%%%%%%%%%%%%%%%%%%%%%%%%%%%%%%%%%%%

The polygonal complex $X=X_{2m,L}$ is that obtained by gluing together all of 
the polygons $P_{ij}$ according to these edge labels, respecting orientation.
Note that the $2m$ edges of each $P_{ij}$ have pairwise distinct labels, so each polygon $P_{ij}$ injects into $X$. We define the group $H=H_{2m,L}$ to be the fundamental group of the polygonal complex $X$.

\begin{lemma}\label{lem:Links} 
The polygonal complex $X=X_{2m,L}$ has $2m$ vertices, and the link of each vertex in $X$ is the graph $L$.
\end{lemma}

\begin{proof}  
Fix $k$ with $1 \leq k \leq m$.  When we glue together all of the $P_{ij}$, since the graph $L$ is connected, all vertices $u_{ij}^k$ will be identified to a single vertex in $X$, say $u^k$, and all vertices $v_{ij}^k$ will be identified to a single vertex in $X$, say $v^k$.  Thus $X$ has $2m$ vertices $u^1,v^1,\dots,u^m,v^m$, which occur in this cyclic order going around the image of any $P_{ij}$ in $X$.  

We claim that at each vertex $u^k$ in $X$, the link is the graph $L$.  By construction, at $u^k$, the incoming edges are $x_1^k, \dots, x_{q_1}^k$ and the outgoing edges are $y_1^k,\dots,y_{q_2}^k$.  So the vertices of $\link(u^k)$ can be identified with the vertex set of $L$.
Now an incoming edge $x_i^k$ is connected to an outgoing edge $y_j^k$ in $\link(u^k)$ if and only if $P_{ij}$ is a face of $X$, which occurs if and only if $(x_i,y_j)$ in an edge of the graph $L$.  The claim follows.  Similarly, $\link(v^k)$ is $L$ for each vertex  $v^k$ in $X$.
\end{proof}

\begin{corollary}  
The universal cover of $X=X_{2m,L}$ is a simply-connected $(2m,L)$-complex.
\end{corollary}

\begin{example}
Let $L$ be the complete bipartite graph $K_{3,4}$. The vertex set is 
$S_1 \sqcup S_2$, where $|S_1|=3$ and $|S_2|=4$. Label the vertices of $S_1$ by
$x_1, x_2, x_3$ and those of $S_2$ by $y_1, y_2, y_3, y_4$.
For each edge in $L$, with endpoints $x_i$ and $y_j$, label the edges of
an $8$-gon cyclically by the word $x_i^1 y_j^1 x_i^2 y_j^2 x_i^3 x_i^3 x_i^4 y_j^4$. 
So, we have twelve polygons corresponding to the following words:

\medskip

\centering
\noindent 
$x_1^1 y_1^1 x_1^2 y_1^2 x_1^3 y_1^3 x_1^4 y_1^4 $ $\quad$ 
$x_2^1 y_1^1 x_2^2 y_1^2 x_2^3 y_1^3 x_2^4 y_1^4 $ $\quad$
$x_3^1 y_1^1 x_3^2 y_1^2 x_3^3 y_1^3 x_3^4 y_1^4 $ 

\noindent 
$x_1^1 y_2^1 x_1^2 y_2^2 x_1^3 y_2^3 x_1^4 y_2^4 $ $\quad$
$x_2^1 y_2^1 x_2^2 y_2^2 x_2^3 y_2^3 x_2^4 y_2^4 $ $\quad$
$x_3^1 y_2^1 x_3^2 y_2^2 x_3^3 y_2^3 x_3^4 y_2^4 $ 

\noindent 
$x_1^1 y_3^1 x_1^2 y_3^2 x_1^3 y_3^3 x_1^4 y_3^4 $ $\quad$
$x_2^1 y_3^1 x_2^2 y_3^2 x_2^3 y_3^3 x_2^4 y_3^4 $ $\quad$
$x_3^1 y_3^1 x_3^2 y_3^2 x_3^3 y_3^3 x_3^4 y_3^4 $ 

\noindent 
$x_1^1 y_4^1 x_1^2 y_4^2 x_1^3 y_4^3 x_1^4 y_4^4 $ $\quad$
$x_2^1 y_4^1 x_2^2 y_4^2 x_2^3 y_4^3 x_2^4 y_4^4 $ $\quad$
$x_3^1 y_4^1 x_3^2 y_4^2 x_3^3 y_4^3 x_3^4 y_4^4 $ 
\end{example}

%%%%%%%%%%%%%%%%%%%%%%%%%%%%%%%%%%%%%%%%%
\subsection{Embedding in the lattice of Bourdon's building}\label{sec:CoveringGamma}
%%%%%%%%%%%%%%%%%%%%%%%%%%%%%%%%%%%%%%%%%

In this section we consider the special case $L = K_{q_1,q_2}$ with $q_1,q_2\geq 2$.  We show that the group $H$ constructed in Section~\ref{sec:H} above embeds with index $q_1 q_2$ in the lattice $\G$ of $I_{2m,L}=I_{2m,\q}$ (see Examples~\ref{egs:Lattices}).

Let $X = X_{2m,L}$ be as constructed in Section~\ref{sec:H} above. Let $\cX$ be the scwol associated to $X$ and $\cP$ be the scwol associated to the polygon~$P$ as in Examples~\ref{egs:Scwols}(1a). Let $H(\cX)$ be the trivial complex of groups over $\cX$ and let $G(\cP)$ be the complex of groups over $\cP$ as constructed in Example~\ref{egs:ComplexesGroups}(2) above. Recall from Examples~\ref{egs:Lattices}(2) that the fundamental group of $G(\cP)$ is the lattice $\G=\G_{2m}(G_1,G_2)$ of $\Aut(I_{2m,L})$.   

Let us label the polygon $P$ underlying $G(\cP)$ as follows. Orient the edges of $P$ cyclically and label the resulting $2m$-cycle going around the boundary of $P$ by the word $e_1 f_1 e_2 f_2 \dots e_m f_m$.  For $k \in \Z/m\Z$, label the vertex of $P$ with incoming edge $e_k$ and outgoing edge $f_k$ by $u_k$, and label the vertex of $P$ with incoming edge $f_k$ and outgoing edge $e_{k+1}$ by $v_k$. Figure \ref{fig:Fig5} shows the labeling of $P$ for $m=3$.

%%%%%%%%%%%%%%%%%%%%%%%%%%%%%%%%%%%%%%%%%
\begin{figure}[t]
	\centering
		\includegraphics{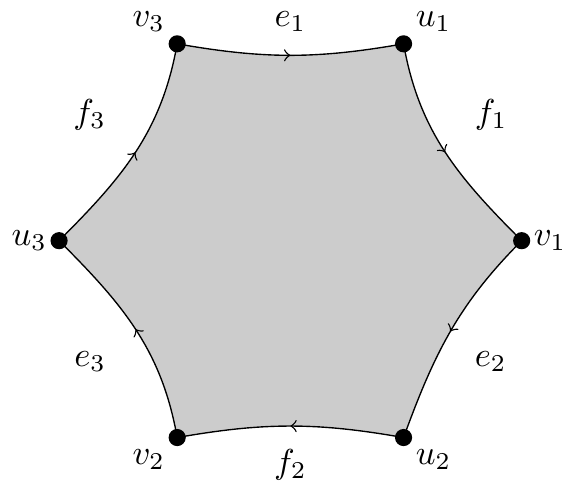}
	\caption{The labeling of $P$} 
\label{fig:Fig5}
\end{figure}
%%%%%%%%%%%%%%%%%%%%%%%%%%%%%%%%%%%%%%%%%

We now prove:

\begin{proposition}\label{prop:CoveringGamma}  
There is a covering of complexes of groups $\Phi:H(\cX) \to G(\cP)$ with $q_1q_2$ sheets. 
\end{proposition}

\begin{proof} 
In addition to notation introduced above, we continue notation from Section~\ref{sec:H}.

We first define a non-degenerate morphism of scwols $p:\cX \to \cP$.  This will be induced by the natural projection $X \to P$. There is an isometry from each polygon $P_{ij}$ to $P$ given by, for $1 \leq k \leq m$, sending the edge $x_i^k$ to $e_k$ and the edge $y_j^k$ to $f_k$, respecting orientations.  Note that the vertex $u_{ij}^k$ maps to $u_k$ and the vertex $v_{ij}^k$  maps to~$v_k$. Since $X$ is obtained from the polygons $P_{ij}$ by gluing them together according to (oriented) edge labels, this collection of isometries $P_{ij} \to P$ induces a projection $p: X \to P$.  By abuse of notation, we write $p:\cX \to \cP$ for the induced map of scwols.  Recall from the proof of Lemma~\ref{lem:Links} that the $2m$ vertices of $X$ are $u^1, v^1, u^2, v^2, \dots, u^m, v^m$.  For each $1 \leq k \leq m$, we have $p(u^k) = u_k$ and $p(v^k) = v_k$.  It is straightforward to verify that $p:\cX \to \cP$ is a non-degenerate morphism of scwols.

We now construct a covering of complexes of groups $\Phi:H(\cX) \to G(\cP)$ over $p:\cX \to \cP$.  For this, we need to construct a family of elements $\{ \phi(a) \in G_{t(p(a))} \mid a \in E(\cX)\}$ which satisfies the conditions in Definition~\ref{def:Covering}.  We begin by labeling the edges of the scwol $\cX$ as follows:
\begin{itemize}
\item {$a_{ij}^k$} goes from the midpoint of $x_i^k$ to the vertex $u^k$;
\item {$b_{ij}^k$} goes from the barycenter of $P_{ij}$ to the midpoint of $x_i^k$;
\item {$c_{ij}^k$} goes from the midpoint of $y_j^k$ to the vertex $u^k$;
\item {$d_{ij}^k$} goes from the barycenter of $P_{ij}$ to the midpoint of $y_j^k$;
\item {$e_{ij}^k$} goes from the midpoint of $x_j^{k+1}$ to the vertex $v^k$; and
\item {$f_{ij}^k$} goes from the midpoint of $y_j^k$ to the vertex $v^k$.
\end{itemize}
Noting carefully the order of composition in Definition~\ref{def:Scwol}, the pairs of edges $(a_{ij}^k, b_{ij}^k)$ and $(c_{ij}^k, d_{ij}^k)$ are composable, with $\alpha_{ij}^k:=a_{ij}^k b_{ij}^k = c_{ij}^k d_{ij}^k$ the edge from the barycenter of $P_{ij}$ to the vertex $u^k$ of $X$.  Similarly, the pairs of edges $(e_{ij}^k, b_{ij}^{k+1})$ and $(f_{ij}^{k}, d_{ij}^k)$ are composable, with $\beta_{ij}^k:=e_{ij}^k b_{ij}^{k+1} = f_{ij}^{k} d_{ij}^k$ the edge from the barycenter of $P_{ij}$ to the vertex $v^k$ of $X$. Figure \ref{fig:Fig6} shows the labeling of (the image of) $P_{ij}$ at the sector between the midpoint of $x^k_i$ and the midpoint of $x^{k+1}_i$. 

%%%%%%%%%%%%%%%%%%%%%%%%%%%%%%%%%%%%%%%%%
\begin{figure}[t]
	\centering
		\includegraphics{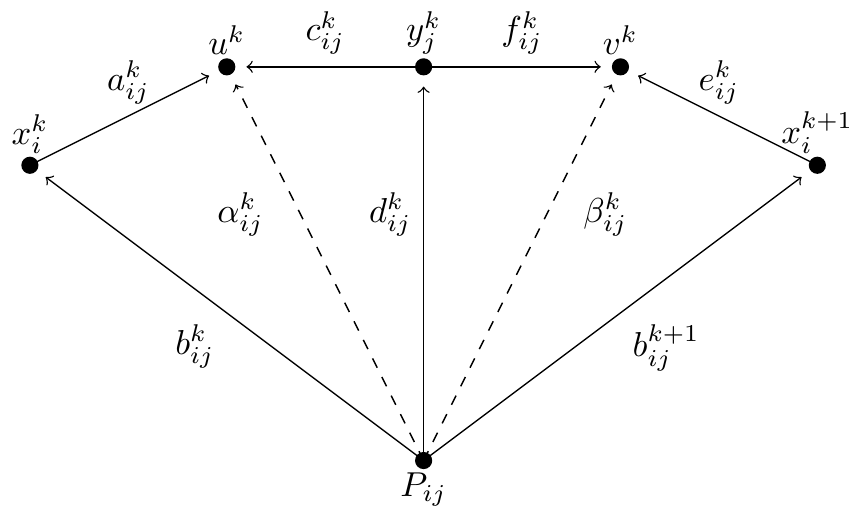}
	\caption{The labeling of (the image of) $P_{ij}$} 
\label{fig:Fig6}
\end{figure}
%%%%%%%%%%%%%%%%%%%%%%%%%%%%%%%%%%%%%%%%%

In the complex of groups $G(\cP)$, for $1 \leq k \leq m$ we specify that the local group at the midpoint of edge $e_k$ is $G_2$, and the local group at the midpoint of edge $f_k$ is $G_1$.  Enumerate the elements of $G_1$ as $\{ g_{1,1}, \dots, g_{1,q_1}\}$ and those of $G_2$ as $\{ g_{2,1}, \dots, g_{2,q_2}\}$.  We identify $G_1$ and $G_2$ with their images in the direct product $G_1 \times G_2$, and write the elements of $G_1 \times G_2$ as $\{ g_{1,i}g_{2,j} \mid 1 \leq i \leq q_1, 1 \leq j \leq q_2 \}$.

We are now ready to define the group element $\phi(a) \in G_{t(p(a))}$, for each $a \in E(\cX)$.  For $1 \leq i \leq q_1$, $1 \leq j \leq q_2$ and $k \in \Z/m\Z$, we put:
\[ \phi(a_{ij}^k) = \phi(d_{ij}^k) = \phi(e_{ij}^k) = g_{1,i} \in G_1, \quad \phi(b_{ij}^k) = \phi(c_{ij}^k) = \phi(f_{ij}^{k-1}) = g_{2,j} \in G_2 \]
and
\[ \phi(\alpha_{ij}^k) = \phi(\beta_{ij}^k) = g_{1,i} g_{2,j} \in G_1 \times G_2.\]
Since the elements $g_{1,i}$ and $g_{2,j}$ commute in $G_1 \times G_2$, Condition (1) of Definition~\ref{def:Covering} holds.  For Condition (2), let $\s \in V(\cX)$ and let $b \in E(\cP)$ be such that $t(b) = p(\s)$.  Then the quotient $G_{t(b)}/\psi_b(G_{i(b)}$ is one of $G_1/\{1\} \cong G_1$, $G_2/\{1\} \cong G_2$, $(G_1 \times G_2)/\{1\} \cong G_1 \times G_2$, $(G_1 \times G_2)/G_1 \cong G_2$ or $(G_1 \times G_2)/G_2 \cong G_1$.  In each case, it may be checked that the elements of the quotient are enumerated by the set $\{ \phi(a) \mid p(a) = b \mbox{ and } t(a) = \s \}$.  Hence Condition (2) holds.  Thus we have constructed a covering of complexes of groups $\Phi:H(\cX) \to G(\cP)$.

Finally, we show that the covering $\Phi$ has $q_1 q_2$ sheets.  Since $L = K_{q_1, q_2}$ has $q_1q_2$ edges, the polygonal complex  $X$ has $q_1 q_2$ faces.  Let $\tau \in V(\cP)$ be the barycenter of the face of $P$.  Then the set $p^{-1}(\tau)$ has $q_1 q_2$ elements, one for each face in $X$.  Now the local group $G_\tau$ is trivial, so the number of sheets of $\Phi$ is $ |p^{-1}(\tau)| = q_1 q_2$, as required.
\end{proof}

We are now able to prove our first main result:

\Embeddingingamma* 

\begin{proof}
Combining Theorem~\ref{thm:Coverings} and Proposition~\ref{prop:CoveringGamma} tells us $H$ is an index $q_1q_2$ subgroup of $\G$. The fact that $H$ is a maximal torsion-free subgroup is a consequence of Corollary~\ref{cor:Maximality}. 
\end{proof}

%%%%%%%%%%%%%%%%%%%%%%%%%%%%%%%%%%%%%%%%%
\subsection{Embedding in the Coxeter group}\label{sec:CoveringW}
%%%%%%%%%%%%%%%%%%%%%%%%%%%%%%%%%%%%%%%%%

In this section we return to the setting where $L$ is any finite, connected, simplicial and bipartite graph. We show that the group $H$ embeds with index $2m$ in the Coxeter group $W$.

Let $X = X_{2m,L}$ and $H = H_{2m,L}$ be as constructed in Section~\ref{sec:H} above. Let $\cXop$ be the opposite scwol associated to $X$, as in Examples~\ref{egs:Scwols}(1b), and let $\cK$ be the scwol associated to the chamber $K$ for $(W_{2m,L},S_L)$, as in Examples~\ref{egs:Scwols}(2). We work with the opposite scwol in this section so that the natural map from the barycentric subdivision of $X$ to $K$ induces a non-degenerate morphism of scwols. Let $H(\cXop)$ be the trivial complex of groups over $\cXop$ and let $G(\cK)$ be the complex of groups over $\cK$ with fundamental group $W=W_{2m,L}$ constructed in Section~\ref{sec:Complexesofgroups}. We now prove:

\begin{proposition}\label{prop:CoveringW}  
There is a covering of complexes of groups $\Psi:H(\cXop) \to G(\cK)$ with $2m$ sheets.
\end{proposition}

\begin{proof}  
We first construct a non-degenerate morphism of scwols $f:\cXop \to \cK$.  For this, we assign types to the vertices of $\cXop$ as follows.  Let $\s \in V(\cXop)$.  If $\s$ is a vertex of $X$, then $\s$ has type $\emptyset$.  If $\s$ is the midpoint of the edge $x_i^k$ of $X$ then $\s$ has type $\{ x_i \}$ and if $\s$ is the midpoint of the edge $y_j^k$ of $X$ then $\s$ has type $\{ y_j \}$.  Finally, if $\s$ is the barycenter of the face $P_{ij}$ of $X$ then $\s$ has type $\{x_i, y_j\}$.  Then for each edge $a \in E(\cXop)$, $i(a)$ is of type $T'$ and $t(a)$ is of type $T$ where $T' \subseteq T$ are spherical subsets of $S = S_L$.  The map $f:\cX \to \cK$ is that induced by sending each vertex of $\cXop$ which has type $T$ to the unique vertex of $\cK$ which has type $T$.  It is straightforward to check that $f$ is a non-degenerate morphism of scwols.

We now construct a covering of complexes of groups $\Psi:H(\cXop) \to G(\cK)$ over $f:\cXop \to \cK$.  We first label the edges of the scwol $\cXop$.  Although $E(\cXop) = E(\cX)$, we will use different labels to those in the proof of Proposition~\ref{prop:CoveringGamma} above.  Our labeling for $\cXop$ is as follows:
\begin{itemize}
\item The edges $A_{ij}^1,\dots,A_{ij}^m$ have terminal vertex the barycenter of $P_{ij}$.  Their initial vertices are the midpoints of the $m$ edges $x_{i}^1, y_j^m, x_j^{m},\dots,y_j^\ell, x_i^{\ell}$, respectively, if $m = 2\ell - 3$ is odd, and the $m$ edges $x_{i}^1, y_j^m, x_j^{m},\dots, x_i^{\ell+1}, y_j^{\ell}$, respectively, if $m = 2\ell - 2$ is even.
\item The edges $B_{ij}^1,B_{ij}^3,\dots,B_{ij}^{2m-1}$ have terminal vertices equal to the initial vertices of $A_{ij}^1,\dots,A_{ij}^m$, respectively.  Their initial vertices are $u^1, v^m, u^m, \dots, v^\ell, u^\ell$, respectively, if $m = 2\ell - 3$ is odd, and $u^1, v^m, u^m, \dots, u^{\ell+1},v^\ell$, respectively, if $m = 2\ell - 2$ is even.
\item The edges $B_{ij}^2,B_{ij}^4,\dots,B_{ij}^{2m}$ have terminal vertices equal to the initial vertices of $A_{ij}^1,\dots,A_{ij}^m$, respectively.  Their initial vertices are $v^m, u^m, \dots, u^\ell,v^{\ell-1}$, respectively, if $m = 2\ell - 3$ is odd, and $v^m, u^m, \dots, v^\ell, u^{\ell}$, respectively, if $m = 2\ell - 2$ is even.
\item The edges $C_{ij}^1,\dots,C_{ij}^m$ have terminal vertex the barycenter of $P_{ij}$.  Their initial vertices are the midpoints of the $m$ edges $y_j^1,x_i^2,y_j^2,\dots,x_i^\ell,y_j^{\ell-1}$, respectively, if $m = 2\ell - 3$ is odd, and the $m$ edges $y_j^1,x_i^2,y_j^2,\dots,y_j^\ell,x_i^\ell$, respectively, if $m = 2\ell - 2$ is even.
\item The edges $D_{ij}^1,D_{ij}^3,\dots,D_{ij}^{2m-1}$ have terminal vertices equal to the initial vertices of $C_{ij}^1,\dots,C_{ij}^m$, respectively.  Their initial vertices are $u^1, v^1, u^2, \dots, u^{\ell-1}$, respectively, if $m = 2\ell - 3$ is odd, and $u^1, v^1, u^2, \dots, v^{\ell-1}$, respectively, if $m = 2\ell - 2$ is even.
\item The edges $D_{ij}^2,D_{ij}^4,\dots,D_{ij}^{2m}$ have terminal vertices equal to the initial vertices of $C_{ij}^1,\dots,C_{ij}^m$, respectively.  Their initial vertices are $v^1, u^2, v^2,\dots, v^{\ell-1}$, respectively, if $m = 2\ell - 3$ is odd, and $v^1, u^2, v^2,\dots, u^\ell$, respectively, if $m = 2\ell - 2$ is even.
\end{itemize}
Noting carefully the order of composition in Definition~\ref{def:Scwol}, the compositions of edges are:
\begin{itemize}
\item two distinguished edges $\varepsilon_{ij}:=A_{ij}^1 B_{ij}^1 = C_{ij}^1 D_{ij}^1$ and $\varepsilon'_{ij}:=A_{ij}^m B_{ij}^{2m} = C_{ij}^m D_{ij}^{2m}$; 
\item for $1 \leq k \leq m-1$, an edge $\gamma_{ij}^k := A^k_{ij} B^{2k}_{ij} = A^{k+1}_{ij} B^{2k+1}_{ij}$; and
\item for $1 \leq k \leq m-1$, an edge  $\delta_{ij}^k := C^k_{ij} D^{2k}_{ij} = C^{k+1}_{ij} D^{2k+1}_{ij}$.
\end{itemize} 

Figure \ref{fig:Fig7} shows the labeling of the (image of the) whole polygon $P_{ij}$ for the case $m=3$. We neglect to label the compositions, but leave them drawn as dotted lines. 

%%%%%%%%%%%%%%%%%%%%%%%%%%%%%%%%%%%%%%%%%
\begin{figure}[t]
	\centering
		\includegraphics{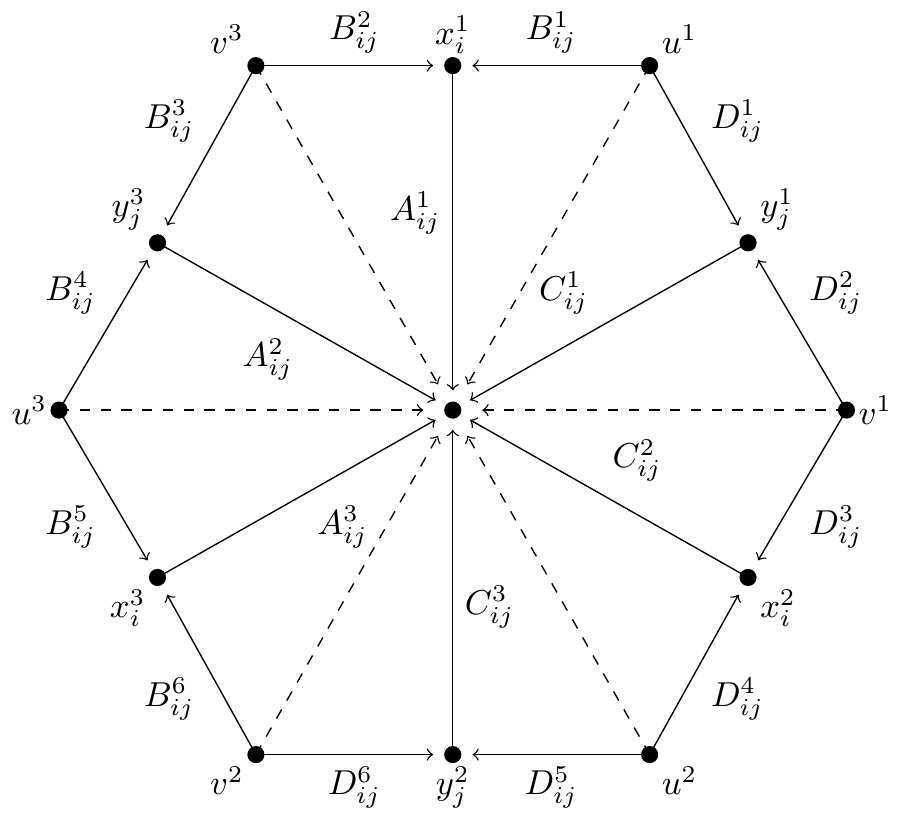}
	\caption{The labeling of (the image of) $P_{ij}$} 
\label{fig:Fig7}
\end{figure}
%%%%%%%%%%%%%%%%%%%%%%%%%%%%%%%%%%%%%%%%%

We next establish some notation and recall some facts concerning finite dihedral groups.  Let $(x_i,y_j)$ be an edge of $L$ and recall that the special subgroup $W_{\{x_i,y_j\}}$ is dihedral of order $2m$.  To simplify notation, write $W_{ij}$ for $W_{\{x_i,y_j\}}$.  For $1 \leq k \leq m$, we denote by $w_k(x_i,y_j)$ the element of $W_{ij}$ given by the alternating product $x_i y_j x_i \dots$ which starts with $x_i$ and has $k$ letters.  Similarly define $w_k(y_j,x_i)$ to start with $y_j$.  Then if $w$ is a nontrivial element of $W_{ij}$, we have $w = w_k(x_i,y_j)$ or $w = w_k(y_j,x_i)$ for some $k$ with $1 \leq k \leq m$, and this expression for $w$ is unique except for the case $w = w_m(x_i,y_j) = w_m(y_j,x_i)$.  For example, if $m = 3$ then the $6$ elements of $W_{ij}$ are $1$, $x_i$, $y_j$, $x_iy_j$, $y_jx_i$ and $x_iy_jx_i = y_jx_iy_j$. 

We are now ready define the family $\{ \phi(A) \in G_{t(f(A))} \mid A \in E(\cXop) \}$.  Recall that in $G(\cK)$, the local group at $G_\s$ is the special subgroup $W_T$ where $\s$ is of type $T$.  We describe the assignment of group elements $\phi(A) \in G_{t(f(A))}$ according to the type of the vertex $t(A) \in V(\cXop)$.
First suppose that $t(A)$ has type $\{x_i\}$.  Then $A$ is the edge $B^{2k-1}$ or $B^{2k}$ with $k$ odd, or $D^{2k-1}$ or $D^{2k}$ with $k$ even.  We put $\phi(B^{2k-1}) = 1$ and $\phi(B^{2k}) = x_i$ if $k$ is odd and $\phi(D^{2k-1}) = 1$ and $\phi(D^{2k}) = x_i$ if $k$ is even.  The assignment is similar if $t(A)$ has type $\{y_j\}$: we put $\phi(B^{2k-1}) = 1$ and $\phi(B^{2k}) = y_j$ if $k$ is even and $\phi(D^{2k-1}) = 1$ and $\phi(D^{2k}) = y_j$ if $k$ is odd.  

Now suppose that $t(A)$ has type $\{x_i,y_j\}$.  Then we put $\phi(A_{ij}^1) = \phi(C_{ij}^1) = 1$, and for $2 \leq k \leq m$ we put $\phi(A_{ij}^k) = w_{k-1}(x_i,y_j)$ and $\phi(C_{ij}^k) = w_{k-1}(y_j,x_i)$.  This takes care of all edges $A$ with $t(A)$ of type $\{x_i,y_j\}$ and $i(A)$ of type either $\{x_i\}$ or $\{y_j\}$.  The remaining edges are the compositions, and so have initial vertex of type $\emptyset$.  We define $\phi(\varepsilon_{ij}) = 1$ and $\phi(\varepsilon'_{ij}) = w_m(x_i,y_j) = w_m(y_j,x_i)$, and for $1 \leq k \leq m-1$ we define $\phi(\gamma_{ij}^k) = w_k(x_i,y_j)$ and $\phi(\delta_{ij}^k) = w_k(y_j,x_i)$. Figure \ref{fig:Fig8} shows the covering locally at the edges $x^2_i$ and $y^2_j$ of $P_{ij}$ for the case $m=3$. Again the color of vertices denotes type. Thicker arrows have been drawn to indicate where a ``folding" takes place.  

%%%%%%%%%%%%%%%%%%%%%%%%%%%%%%%%%%%%%%%%%
\begin{figure}[t]
	\centering
		\includegraphics{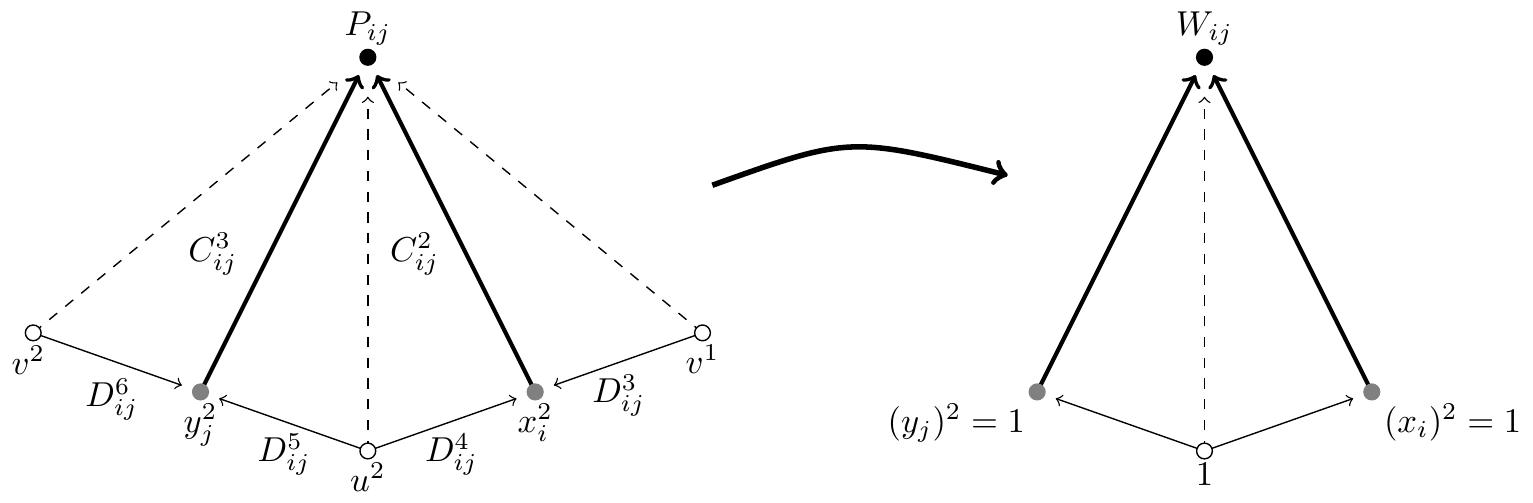}
	\caption{The covering of $G(\cK)$ by $\cXop$} 
\label{fig:Fig8}
\end{figure}
%%%%%%%%%%%%%%%%%%%%%%%%%%%%%%%%%%%%%%%%%

The verification of Condition (1) in Definition~\ref{def:Covering} is straightforward, keeping in mind the order of composition of edges in Definition~\ref{def:Scwol}, since for $m \geq 3$ the elements $x_i$ and $y_j$ do not commute in $W_{ij}$.  For example, we have $A^2_{ij}B^4_{ij} = \gamma^2_{ij}$, and $\phi(A^2_{ij}) = x_i$, $\phi(B^4_{ij}) = y_j$ and $\phi(\gamma^2_{ij}) = x_i y_j$.  

For Condition (2) in Definition~\ref{def:Covering}, let $\s \in V(\cXop)$ and let  $b \in E(\cK)$ be such that $t(b) = f(\s)$.  Then $\s$ has type $\{x_i\}$, $\{y_j\}$ or $\{x_i,y_j\}$.  If $\s$ has type $\{x_i\}$, there are exactly two edges in $\cXop$ which have terminal vertex $\s$, say $A$ and $A'$.  These are both mapped by $f$ to $b$, and by construction the sets $\{ \phi(A), \phi(A') \}$ and $\{1,x_i \}$ are equal.  Hence $A \mapsto \phi(A)$ induces the desired bijection from $\{ A \in E(\cXop) \mid f(A) = b \mbox{ and } t(A) = \s \}$ to $G_{f(\s)}/\psi_b(G_{i(b)} = W_{\{x_i\}}/W_\emptyset \cong W_{\{x_i\}} = \{1,x_i\}$.  The argument is similar if $\s$ has type $\{y_j\}$.

If $\s$ has type $\{x_i,y_j\}$ we consider cases according to the type of $i(b)$.  Suppose $i(b)$ has type $\{x_i\}$.  By definition of $f$, the set $\{ A \in E(\cXop) \mid f(A) = b \mbox{ and } t(A) = \s \}$ is equal to 
\[
\{ A \in E(\cXop) \mid i(A) \in \{ x_i^1,\dots,x_i^m\} \mbox{ and }t(A) = \s\}.
\]
It is then not difficult to check that the map $A \mapsto \phi(A)$ induces a bijection from this set 
to a set of representatives of the left cosets of $W_{\{x_i\}} = \{1,x_i\}$ in the dihedral group $W_{ij}$ of order $2m
$.  For example, if $m = 3$ then the edges $A_{ij}^1$, $C_{ij}^2$ and $A_{ij}^3$ have initial vertices the midpoints of $x_i^1$, $x_i^2$ and $x_i^3$, respectively, and we have $\phi(A_{ij}^1) = 1$, $\phi(C_{ij}^2) = y_j$ and $\phi(A_{ij}^3) = x_i y_j$, while the left cosets of $\{1,x_i\}$ in $W_{ij}$ are $\{1,x_i\}$, $\{y_j, y_j x_i\}$ and $\{ x_i y_j, x_i y_j x_i\}$.  Condition (2) then follows.  The argument is similar if $i(b)$ has type $\{y_j\}$.  

The final case is when $\s$ has type $\{x_i, y_j\}$ and $i(b)$ has type $\emptyset$.  Then by construction, the map $A \mapsto \phi(A)$ is a bijection from the set of composition edges 
\[
\{ \varepsilon_{ij}, \varepsilon'_{ij}\} \cup \{ \gamma^k_{ij}, \delta^k_{ij} \mid 1 \leq k \leq m-1 \}
\] 
to the elements of $W_{ij}$.  Thus in this case Condition (2) holds as well.  We have now constructed a covering of complexes of groups $\Psi:H(\cXop) \to G(\cK)$.

We conclude by showing that the covering $\Psi$ has $2m$ sheets.  For this, recall that the polygonal complex $X$ has $2m$ vertices. Let $\tau \in V(\cK)$ be the cone point of $K$.  Then the set $f^{-1}(\tau)$ consists of all vertices of $\cXop$ which have type $\emptyset$, and these are exactly the $2m$ vertices of $X$.   Since the local group $G_\tau$ is trivial, the number of sheets of $\Psi$ is $ |f^{-1}(\tau)| = 2m$, as required.
\end{proof}

Using Proposition~\ref{prop:CoveringW} and the same covering-theoretic results as in Section~\ref{sec:CoveringGamma} above, we obtain our second main result:

\EmbeddinginW*
 
We also note that by the last statement in Theorem~\ref{thm:Coverings}, the covering $\Psi$ induces an isomorphism between the universal covers of $H(\cXop)$ and $G(\cK)$.  Hence:

\begin{corollary}\label{daviscomplexisbuilding}  
Let $\widetilde X$ be the universal cover of the polygonal complex $X = X_{2m,L}$.  Then the barycentric subdivision of $\widetilde X$ is simplicially isomorphic to the Davis complex $\Sigma=\Sigma_{2m,L}$.
\end{corollary}

%%%%%%%%%%%%%%%%%%%%%%%%%%%%%%%%%%%%%%%%%%%%%%%%
%%%%%%%%%%%%%%%%%%%%%%%%%%%%%%%%%%%%%%%%%%%%%%%%
\section{Description as an amalgam of surface groups over free groups}\label{sec:Amalgam}
%%%%%%%%%%%%%%%%%%%%%%%%%%%%%%%%%%%%%%%%%%%%%%%%
%%%%%%%%%%%%%%%%%%%%%%%%%%%%%%%%%%%%%%%%%%%%%%%%

In our final section we describe $H$ as an amalgam of genus $(m-1)$ surface groups over rank $(m-1)$ free groups for the special case $L$ is a complete bipartite graph. 

%%%%%%%%%%%%%%%%%%%%%%%%%%%%%%%%%%%%%%%%%
\subsection{A presentation for the torsion-free group}\label{sec:Presentation}
%%%%%%%%%%%%%%%%%%%%%%%%%%%%%%%%%%%%%%%%%

Let $L = K_{q_1,q_2}$ with $q_1,q_2\geq 2$. First we obtain a group presentation for $H=H_{2m,L}$. Form a polygonal complex homotopic to $X=X_{2m,L}$ by contracting the polygon labeled
\[
x_{q_1}^1 y_{q_2}^1 x_{q_1}^2 y_{q_2}^2 \dots x_{q_1}^m y_{q_2}^m 
\]
to a single vertex. The resulting complex has a single vertex, allowing us to obtain a presentation for its fundamental group by reading off relations from the boundaries of the remaining 2-dimensional cells (see~\cite{Munk} Theorem 72.1). We have two types of relations. The first type correspond to those polygons in $X$ which didn't contain edges in common with the collapsed polygon. These are of the form
\[
x_{i}^1 y_{j}^1 x_{i}^2 y_{j}^2 \dots x_{i}^m y_{j}^m=1
\]
for $i\in \{1,2, \dots, q_1-1\}$, $j\in \{1,2, \dots, q_2-1\}$. The second type correspond to those polygons which contained contracted edges. After contracting, their edge labels are of the form 
\[
x_{i}^1 x_{i}^2 \dots x_{i}^m=1
\]
for $i\in \{1,2, \dots, q_1-1\}$, or 
\[
y_{j}^1 y_{j}^2 \dots y_{j}^m=1
\]
for $j\in \{1,2, \dots, q_2-1\}$. 

\begin{proposition}\label{prop:pres}
The group $H$ is generated by $x_i^1,x_i^2\dots,x_i^{m-1},y_j^1,y_j^2\dots,y_j^{m-1}$ subject to the following $(q_1 - 1)(q_2 - 1)$ many relations: 
\[
x_{i}^1 y_{j}^1 x_{i}^2 y_{j}^2 \dots x_{i}^{m-1} y_{j}^{m-1}    (x_{i}^{m-1})^{-1} (x_{i}^{m-2})^{-1} \dots (x_{i}^{1})^{-1}  (y_{j}^{m-1})^{-1} (y_{j}^{m-2})^{-1} \dots (y_{j}^{1})^{-1}=1
\] 
for $i\in \{1,2, \dots, q_1-1\}$, $j\in \{1,2, \dots, q_2-1\}$.
\end{proposition} 

\begin{proof}
Let us rearrange the relations of the second type to give
\[
x_{i}^m=(x_{i}^{m-1})^{-1} (x_{i}^{m-2})^{-1} \dots (x_{i}^{1})^{-1}
\]
for $i\in \{1,2, \dots, q_1-1\}$, and
\[
y_{j}^m=(y_{j}^{m-1})^{-1} (y_{j}^{m-2})^{-1} \dots (y_{j}^{1})^{-1}
\]
for $j\in \{1,2, \dots, q_2-1\}$. We can then substitute these expressions for $x_{i}^m$ and $y_{j}^m$ into the relations of the first type to obtain the required presentation.
\end{proof}

%%%%%%%%%%%%%%%%%%%%%%%%%%%%%%%%%%%%%%%%%
\subsection{Amalgams of surface groups}\label{sec:SurfaceGroups}
%%%%%%%%%%%%%%%%%%%%%%%%%%%%%%%%%%%%%%%%%

We now show that our presentation for $H$ can be recognized as a presentation of an amalgam of surface groups. 

For fixed $i\in \{1,2, \dots, q_1-1\}$ and $j\in \{1,2, \dots, q_2-1\}$, we denote the group generated by $x_{i}^1,x_{i}^2\dots,x_{i}^{m-1},y_{j}^1,y_{j}^2\dots,y_{j}^{m-1}$ subject to the single relation
\[
x_{i}^1 y_{j}^1 x_{i}^2 y_{j}^2 \dots x_{i}^{m-1} y_{j}^{m-1}    (x_{i}^{m-1})^{-1} (x_{i}^{m-2})^{-1} \dots (x_{i}^{1})^{-1}  (y_{j}^{m-1})^{-1} (y_{j}^{m-2})^{-1} \dots (y_{j}^{1})^{-1}=1
\]  
by $S_{m-1}\langle x_{i},y_{j}\rangle$. 

\begin{lemma}\label{lem:surface}
The group $S_{m-1}\langle x_{i},y_{j}\rangle$ is a surface group of genus $(m-1)$. 
\end{lemma}

\begin{proof}
Write the group's relation on the boundary of a $4(m-1)$-gon. The corresponding quotient space is a genus $(m-1)$ surface with a single vertex, allowing us to obtain a presentation for its fundamental group by reading off a single relation from the boundary of its 2-dimensional cell. This gives the group $S_{m-1}\langle x_{i},y_{j}\rangle$.
\end{proof}

Let us denote the free group with basis $x_{i}^1, x_{i}^2,  \dots, x_{i}^{m-1}$ by $F_{m-1}\langle x_i\rangle$ and the free group with basis $y_{i}^1, y_{i}^2,  \dots, y_{i}^{m-1}$ by $F_{m-1}\langle y_j\rangle$. We denote the embedding $F_{m-1}\langle x_i\rangle \hookrightarrow S_{m-1}\langle x_i,y_j\rangle$ such that $x_i \mapsto x_i$ by $\iota_{m-1}(x_i,y_j)$, and the embedding $F_{m-1}\langle y_j\rangle \hookrightarrow S_{m-1}\langle x_i,y_j\rangle$ such that $y_j \mapsto y_j$ by $\iota_{m-1}(y_j,x_i)$. 

We now describe how to associate a scwol to a simplicial graph $\Theta$. Let $\Theta'$ denote the barycentric subdivision of $\Theta$. We associate to $\Theta$ the scwol $\cX=\cX(\Theta)$ such that $V(\cX):= V(\Theta')$ and $E(\cX) := E(\Theta')$. Each edge $a\in E(\cX)$ is then oriented so that $i(a)$ is a vertex of $\Theta$ and $t(a)$ is the barycenter of an edge of $\Theta$. Let us denote $\cX(K_{q_1-1,q_2-1})$ by $\cX_{q_1-1,q_2-1}$. Figure \ref{fig:Fig9} shows $\cX_{1,2}$.   

%%%%%%%%%%%%%%%%%%%%%%%%%%%%%%%%%%%%%%%%%
\begin{figure}
\centering
\begin{minipage}{.5\textwidth}
  \centering
  \includegraphics{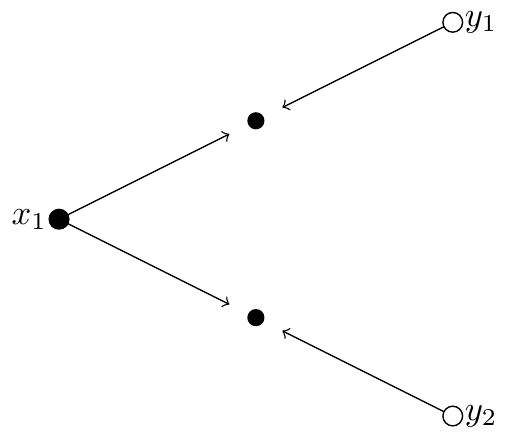}
  \caption{The scwol $\cX_{1,2}$}
  \label{fig:Fig9}
\end{minipage}%
\begin{minipage}{.5\textwidth}
  \centering
  \includegraphics{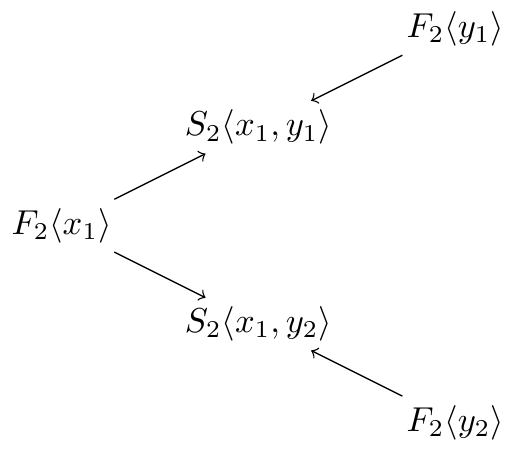}
  \caption{The complex of groups $G_{2}(\cX_{1,2})$} 
  \label{fig:Fig10}
\end{minipage}
\end{figure}
%%%%%%%%%%%%%%%%%%%%%%%%%%%%%%%%%%%%%%%%%

The scwol $\cX(\Theta)$ is an example of a 1-dimensional scwol (see Section \ref{sec:Scwols}), i.e. a scwol which has no composable pairs of edges. Scwols which are 1-dimensional have two kinds of vertices; \emph{sources} are initial vertices of edges and \emph{sinks} are terminal vertices of edges. Complexes of groups over 1-dimensional scwols are (equivalent to) graphs of groups in the sense of Bass-Serre~\cite{Serre}.  

We now use the groups $S_{m-1}\langle x_i,y_j\rangle$, $F_{m-1}\langle x_i\rangle$ and $F_{m-1}\langle y_i\rangle$ to construct a complex of groups $G_{m-1}(\cX) = (G_\s,\psi_a)$ over $\cX=\cX_{q_1-1,q_2-1}$. Our construction can probably be deduced from Figure \ref{fig:Fig10}, which shows $G_{2}(\cX_{1,2})$. 

Explicitly let $\s \in V(\cX)$. If $\s$ is the barycenter of $(x_i,y_j)$, the local group $G_\s$ is  $S_{m-1}\langle x_i,y_j\rangle$. If $\s$ is the vertex $x_i$, the local group $G_\s$ is $F_{m-1}\langle x_i\rangle$. If $\s$ is the vertex $y_j$, the local group $G_\s$ is $F_{m-1}\langle y_j\rangle$. Let $a \in E(\cX)$. If $i(a)=x_i$ and $t(a)=(x_i,y_j)$ then $\psi_a=\iota_{m-1}(x_i,y_j)$. If $i(a)=y_j$ and $t(a)=(x_i,y_j)$ then $\psi_a=\iota_{m-1}(y_j,x_i)$. We now prove our final main result.

\Amalgam* 

\begin{proof}
We claim the direct limit (i.e. the colimit) of $G_{m-1}(\cX)$ as a diagram of groups is isomorphic to $H$. Direct limits of diagrams of groups can easily be constructed by generators and relations (see~\cite{Serre}, p.1). Since $G_{m-1}(\cX)$ is a diagram over a 1-dimensional scwol which contains only embeddings, the construction can be simplified as follows. First one takes as generators and relations the disjoint union of sets of generators and relations for each $G_\s$ such that $\s \in V(\cX)$ is a sink. One then identifies each pair of generators $h,h'$ such that there exist $\psi_a$ and $\psi_b$ with $i(a)=i(b)=\s$, and $g\in G_\s$ with $\psi_a(g)=h$ and $\psi_b(g)=h'$. 

Each local group at a sink of $G_{m-1}(\cX)$ is a surface group $S_{m-1}\langle x_{i},y_{j}\rangle$ which comes equipped with the group presentation given above. By the way we have constructed our local groups, the effect of taking the disjoint union of generators and then identifying generators in the required manner is the same as simply taking their union. Hence the direct limit of $G_{m-1}(\cX)$ is the group generated by $x_{i}^1,x_{i}^2\dots,x_{i}^{m-1},y_{j}^1,y_{j}^2\dots,y_{j}^{m-1}$ subject to the relations 
\[
x_{i}^1 y_{j}^1 x_{i}^2 y_{j}^2 \dots x_{i}^{m-1} y_{j}^{m-1}    (x_{i}^{m-1})^{-1} (x_{i}^{m-2})^{-1} \dots (x_{i}^{1})^{-1}  (y_{j}^{m-1})^{-1} (y_{j}^{m-2})^{-1} \dots (y_{j}^{1})^{-1}=1
\]
for $i\in \{1,2, \dots, q_1-1\}$, $j\in \{1,2, \dots, q_2-1\}$. This coincides with the presentation of $H$ found in Proposition~\ref{prop:pres}.    
\end{proof} 

%%%%%%%%%%%%%%%%%%%%%%%%%%%%%%%%%%%%%%%%%
\subsection{Right-angled Artin groups}\label{sec:RAAGs}
%%%%%%%%%%%%%%%%%%%%%%%%%%%%%%%%%%%%%%%%%

If $m=2$ we recover a special case of the following construction. A \emph{right-angled Artin group} is a group $A_\Theta$ associated to a simplicial graph $\Theta$ in the following way: one takes the free group on the vertices of $\Theta$, and then quotients out the commutators of any adjacent vertices. The group $A_\Theta$ is then the direct limit of a diagram of groups over $\cX(\Theta)$, where sink groups are free abelian groups of rank 2, and source groups are infinite cyclic groups. Therefore if $A_\Theta$ is connected and contains at least two edges, then $A_\Theta$ is a non-trivial amalgam of free abelian groups of rank 2 over infinite cyclic groups. In this sense $H$, for $m\geq 3$, can be viewed as a generalization of right-angled Artin groups associated to complete bipartite graphs to higher genus.   

%%%%%%%%%%%%%%%%%%%%%%%%%%%%%%%%%%%%%%%%%
\subsection{Geometric amalgams}\label{sec:Geoamalgams}
%%%%%%%%%%%%%%%%%%%%%%%%%%%%%%%%%%%%%%%%%

In~\cite{Lafont} Lafont introduces a family of diagrams of free groups equipped with some geometric data, and a family of geodesic metric spaces such that the fundamental group functor $\pi_1$ induces a bijection between (the isomorphism classes of) the two families. The diagrams are 1-dimensional scwols such that sources have valency at least three, and are populated with groups by putting free groups of rank 1 on sources, and free groups of rank $\geq 2$ on sinks. The metric spaces are amalgams of hyperbolic surfaces-with-boundary over totally geodesic loops. Lafont proves $\pi_1$ rigidity for the metric spaces, giving ``diagram rigidity" (i.e. colimit rigidity) for the diagrams of free groups. 

In the spirit of Lafont one can construct a $K(G,1)$ for $H$ by gluing together $(q_1 - 1)(q_2 - 1)$ many $4(m-1)$-gons according to their labeling by our surface group relations. The resulting space is an amalgam of closed genus $(m-1)$ surfaces over bouquets of $(m-1)$ many loops.
 
\bibliography{Torsion-freesubgroups} 

\end{document}